\documentclass[a4paper,10pt,leqno]{article}
\usepackage{amsmath,amsthm,amssymb}
\usepackage{graphicx}

\newtheorem{proposition}{Proposition}[section]

\newtheorem{theorem}[proposition]{Theorem}

\newtheorem{lemma}[proposition]{Lemma}

\newtheorem{remark}[proposition]{Remark}

\numberwithin{equation}{section}

\newcommand{\R}{\mathbf{R}}

\newcommand{\supp}{\mathrm{supp}}

\newcommand{\8}{\infty}
\newcommand{\pl}{\partial}

\begin{document}

\title{Remarks on the asymptotic behavior of the solution to an abstract damped wave equation}
\author{Hisashi Nishiyama}
\date{}
\maketitle

\begin{abstract}
We study an abstract damped wave equation. We prove that the solution of the damped wave equation becomes closer to the solution of a heat type equation as time tend to infinity. As an application of our approach, we also study the asymptotic behavior of the damped wave equation in Euclidean space under the geometric control condition.
\\

\\

\noindent
AMS SUBJECT CLASSIFICATION: 35L05, 35K05, 35 Q99, 46N20
\end{abstract}

\begin{section}{Introduction}
We study the following abstract damped wave equation 
\begin{equation}\label{abstract wave}
(\pl_t^2+A+B\pl_t)u=0.
\end{equation}
Here  $\pl_t=\frac{\pl}{\pl t}$ and $A$, $B$ are self-adjoint operators on a Hilbert space $H$.  We assume $A$ is a densely-defined positive  operator i.e. $A>0$. $B$ is a strictly positive bounded operator on $H$ i.e. $B\geq c\ Id$ for a constant $c>0$. We also assume $B^{1/2}$ and $B^{-1/2}$ are continuous operators on $D(A)$, i.e. there exists  a constant $C>0$ such that for any $u\in D(A)$, the following bounds hold
\begin{align*}
&\|A B^{1/2}u\|_H \leq C (\|u\|_H+\|A u\|_H),\\
&\|A B^{-1/2}u\|_H \leq C (\|u\|_H+\|A u\|_H).
\end{align*}
An example of this equation is the usual damped wave equation
\begin{equation}\label{damped wave}
(\pl_t^2-\Delta+a(x)\pl_t)u=0\  \text {in} \ \R\times \R^d
\end{equation}
where $\Delta$ is the constant coefficient Laplace operator on $\R^d$, $a(x)$ is a strictly positive definite smooth function with bounded derivatives. It is known that eventually the solution of $(\ref{damped wave})$ are close to the solution of a heat type equation. Here we see this by heuristic argument. The stationary equation is closely related to the asymptotic profile. So we take $u=e^{t\lambda}u_0(x)$ in $(\ref{damped wave})$, we have 
the following stationary damped wave equation
\begin{equation}\label{stationary damped wave}
(\lambda^2+a(x)\lambda-\Delta)u_0=0.
\end{equation}
To the solution of $(\ref{damped wave})$, the influence of large frequency part becomes smaller as time tend to infinity by damping. 
So for small $\lambda$, if we neglect $\lambda^2$ term in $(\ref{stationary damped wave})$, we obtain
$$(a(x)\lambda-\Delta)u_0=0.$$
This equation is a stationary equation to the following heat equation
$$(a(x)\pl_t-\Delta)u_0=0.$$
So we may expect the asymptotic behavior of the solution to $(\ref{abstract wave})$ is close to the following abstract heat equation
\begin{equation}\label{abstract heat}
B^{1/2}(\pl_t+\tilde A)B^{1/2}v=0.
\end{equation}
Here $\tilde A=B^{-1/2}AB^{-1/2}$. The solution of this equation can be written as follows
$$B^{-1/2}e^{-t\tilde A}B^{1/2}u_0. $$
Thus we can expect that the solution of the damped wave equation approach to the above solution as time tend to infinity.  The following statement gives a justification of this argument.
\begin{theorem}\label{main}
Let $u$ be the solution to the Cauchy problem of the following abstract damped wave equation
\begin{equation}\label{Cauchy}
\begin{cases}
&(\pl_t^2+A+B\pl_t)u=0,\\
& u|_{t=0}=u_0, \pl_t u|_{t=0}=u_1.
\end{cases}
\end{equation}
Then there exists $C>0$ such that for any $u_0\in D(\sqrt A)$, $u_1\in H$ and $t>1$, the following asymptotic profile holds
\begin{align*}
\|u(t)-B^{-1/2}e^{-t\tilde A}B^{1/2}u_0-B^{-1/2}e^{-t\tilde A}B^{-1/2}u_1\|_H\leq C{t^{-1}}(\|u_0\|_H+\|\sqrt Au_0\|_H+\|u_1\|_H ).
\end{align*}
\end{theorem}
\begin{remark}
We can get more sharp estimate, see Remark $\ref{improve}$. On the other hand, this decay rate is optimal, see Proposition $\ref{optimal}$ which is a generalization of the result of \cite{Ch-Ha}. 
\end{remark}
In many situation, we impose some assumptions to the initial data. In the case, Theorem ${\ref{main}}$ does not give sufficient information. 
For example, we consider the constant coefficient damped wave equation on $\R^d$
$$(\pl_t^2 -\Delta +\pl_t )u=0.$$
The correspondence heat propagator satisfies
$$e^{t\Delta}: L^1 \to L^2={\cal O}(t^{-d/4}).$$
From this decay estimate, we can not get the top term from Theorem ${\ref{main}}$ if the initial date $u_0,u_1$ are in $L^1$ and $d>3$.  So 
we need to impose some more assumption to the equation. For the purpose, we specify $H=L^2(\Omega ,\mu)$ where $(\Omega,\mu)$ is a $\sigma$-finite measure space and assume that $\{e^{-t\tilde A}\}_{t>0}$ is also 
a bounded semi-group on $L^1(\Omega,\mu)$
\begin{equation}\label{bddsemi}
\|e^{-t\tilde A}u\|_{L^1}={\cal O}(1)\|u\|_{L^1}, t>0.
\end{equation}
We also assume that there exists $m>0$ such that the following diffusion estimate holds
\begin{equation}\label{diff}
\|e^{-t\tilde A}u\|_{L^2}={\cal O}(t^{-m})\|u\|_{L^1},\ t>0.
\end{equation}
Furthermore $B^{1/2}$ is also a bounded operator on $L^1$ satisfying
\begin{equation}\label{B}
\frac{1}{C} \|u\|_{L^1}\leq \|B^{1/2}u\|_{L^1}\leq C\|u\|_{L^1}
\end{equation}
for a constant $C>0$. Our model problem $(\ref{damped wave}) $, satisfies this assumption. The next result is a diffusion type estimate for such operators. 
\begin{theorem}\label{main 2}
Under the assumptions of Theorem ${\ref{main}}$, we also assume $(\ref{bddsemi})$, $(\ref{diff})$ and $(\ref{B})$. Then there exists a constant $C>0$ such that for any $u_0\in D(\sqrt A)\cap L^1$, $u_1\in L^2 \cap L^1$, $t>1$, the solution $u(t)$ of $(\ref{Cauchy})$ satisfies the following asymptotic profile
\begin{align*}
\|u(t)-B^{-1/2}e^{-t\tilde A}B^{1/2}u_0-B^{-1/2}e^{-t\tilde A}B^{-1/2}u_1\|_{L^2}\leq C{t^{-m-1}}(\|u_0\|_{L^2 \cap L^1}+\|\sqrt Au_0\|_{L^2}+\|u_1\|_{L^2 \cap L^1} ).
\end{align*}
\end{theorem}
 In the final section, as an application of the above theorems, we consider an example of damped wave equation on Euclidean space and we give an asymptotic profile of the solution for a variable coefficient damped wave equation, see Theorem $\ref{euc damp}$. We also treat a perturbation of this case for which the strictly positivity of the damping term may not be satisfied but the geometric control condition holds. In this case, we can give the similar asymptotic profile,  
see Theorem $\ref{geom asym}$.

Now we remark some related result. 
On bounded domains or manifolds, under the geometric control condition, the exponential energy decay of the solution was proved in \cite{Ba-Le-Ra1}, \cite{Ra-Ta}. On the other hand, for unbounded domains the energy does not decay exponentially. 
In \cite{Ma}, the fact that the decay late of the solution is similar to the heat equation is proved to study the semi-linear damped wave equation. The solution of the damped wave equation tend to the solution of the heat equation is well-known for example \cite{Ni}. This type results  may be called diffusion phenomena for the damped wave equation. For variable coefficient damping, Y, Wakasugi 
(\cite{Wa}) proved the diffusion phenomena for spacial decreasing damping. For abstract setting, the diffusion phenomena may be proved in \cite{Ra-To-Yo} but at present, the manuscript does not published and their result seems to slightly different from ours.  

If the damping term is too weak, for example short range case, the asymptotic behavior of the solution is quite different. In the case, the solution tend to the solution of the wave equation and the local energy decay is an important problem, see \cite{Bo-Ro}.
\end{section} 
\begin{section}{Abstract damped wave equation}
We recall the setting of the problem. In this section, we only assume that $B$ is a bounded and non-negative operator. We reduce $(\ref{abstract wave})$ to a first order system. We introduce the following operator
$${\cal A}=
\begin{pmatrix}
0 & Id \\
- A& -B \\
\end{pmatrix}
:\ {\cal H}_0\rightarrow {\cal H}_0,\  D({\cal A})=\{{}^t(u,v)\in H_1\times H; Au\in H, v\in D(\sqrt A)\}.
$$
Here ${\cal H}_0=H_1\times H$ is the energy space, $H_1$ is the Hilbert space completion of $D(A)$ by the norm
$$\|u\|_{H_1}=\|{\sqrt A } u\|_H^2.$$
$H^{-1}$ denotes the dual space of $H_1$. We note $D(\sqrt A)$ is naturally  embedded in $H^1$. So in what follows, we consider $D(\sqrt A)$ is a subspace of $H^1$. By using this operator, we can rewrite the wave equation (\ref{Cauchy}) as follows, 
\begin{equation}
\pl_t
\begin{pmatrix}
u\\ \pl_t u\\
\end{pmatrix}={\cal A}
\begin{pmatrix}
u\\ \pl_t u\\
\end{pmatrix},\ \ 
\left.\begin{pmatrix}
u\\ \pl_t u\\
\end{pmatrix}\right|_{t=0}
=\begin{pmatrix}
u_0\\ u_1
\end{pmatrix}
\in D(\sqrt A)\times H.
\end{equation}
Since we are considering the damped wave equation, we have the following proposition.
\begin{proposition}
${\cal A}$ is a m-dissipative operator.
\end{proposition}
\begin{proof}
From the definition, we easily prove $\cal A$ is closed. For $(u,v)\in D(A)$, we compute as follows
\begin{align*}
{\rm Re }({\cal A}{}^t(u,v), {}^t(u,v))_{{\cal H}_0}
={\rm Re }\{(A^{1/2}v,A^{1/2}u)_H-(Au+Bv,v)_H\}=-(Bv,v)_H\leq0.
\end{align*}
Thus ${\cal A}$ is dissipative.
\end{proof}
From this proposition, we have the following resolvent estimate for ${\rm Re} \lambda>0$
$$\|(\lambda-{\cal A})^{-1}\|_{{\cal H}_0\to {\cal H}_0}\leq\frac{1}{{\rm Re} \lambda}. $$
So by Hille-Yosida Theorem, $\cal A$ generates a contraction semi-group
$$U(t)=e^{t{\cal A}}: {{\cal H}_0}\rightarrow  {{\cal H}_0}.$$
To an initial date ${}^t(u_0,u_1)\in {\cal H}_0$, we write the solution as follows
$${}^t(u(t),v(t))=U(t){}^t(u_0,u_1).$$  
We restrict the initial date since the energy space ${\cal H}_0$ is a exotic space. For ${}^t(u_0,u_1)\in D(A) \times D(\sqrt A)\subset D({\cal A})$, since $U(t){}^t(u_0,u_1)$ is $C^1$, we have the following identity in $H^1$
$$u(t)=u_0+\int_0^tv(s)ds.$$
By the density argument, the above identity also holds for ${}^t(u_0,u_1)\in D(\sqrt A) \times H$. From the right hand side of the above identity.,we can regard $u(t)\in H$. So we introduce the space ${\cal H}= D(\sqrt A)\times H$ by the norm 
$$\|(u, v)\|_{{\cal H}}=\{\|u\|^2_{H}+\|u\|^2_{H_1}+\|v\|^2_H\}^{1/2}.$$
Since $U(t)$ is a contraction operator, from the above identity, we have the following bound for the propagator
\begin{equation}\label{propa}
\|U(t)\|_{{\cal H}\to {\cal H}}\leq C(1+t).
\end{equation}
In this paper, we call $u(t)\in D(\sqrt A)$ for ${}^t(u(t),v(t))=U(t){}^t(u_0, u_1)$, $(u_0, u_1)\in {\cal H}$ as the solution to (\ref{Cauchy}).

\end{section}
\begin{section}{Resolvent estimate}
We prove some related estimates to the resolvent of ${\cal A}$. Probably these estimates are well-known, c.f. \cite{Bu-Hi}, but for completeness, we give the proof. 
First we study the following operator which is almost similar to the resolvent
\begin{equation}
R(\lambda)=(\lambda^2+B\lambda +A)^{-1}.
\end{equation}
From the positivity of $B$, we have the following estimate.
\begin{lemma}\label{lem3.1}
There exists $\delta>0$ such that for ${\rm Re}\lambda>-\delta$ and ${\rm Im}\lambda\not=0$, $R(\lambda): H\rightarrow D(A)$ exists and the following bound holds
\begin{equation}\label{eq2}
\|R(\lambda)\|_{ H\to H} ={\cal O}(|{\rm Im}\lambda|^{-1}).
\end{equation}
\end{lemma}
\begin{proof} For $u\in D(A)$, we have
\begin{align}\label{eq1}
((\lambda^2+B\lambda +A) u, u)_H&=((|{\rm Re}\lambda|^2-|{\rm Im}\lambda|^2+B{\rm Re}\lambda + A)u,u)_H+i{\rm Im}\lambda((2{\rm Re}\lambda+B)u,u)_H.
\end{align}
Taking imaginary part of the above identity and using $B\geq cId$ we have the following bound
\begin{align*}
|((\lambda^2+B\lambda +A) u, u)_H|&\geq|{\rm Im}\lambda((2{\rm Re}\lambda+B)u,u)_H|\\
&\geq|{\rm Im}\lambda|(c+2{\rm Re}\lambda)(u,u)_H
\end{align*}
for ${\rm Re}\lambda>-c/2$. So $(\lambda^2+B\lambda +A)$ is injective if ${\rm Re}\lambda$ is sufficiently small and ${\rm Im}\lambda\not=0$. We can also prove $(\lambda^2+B\lambda +A)$ is surjective by considering adjoint operator $(\overline{\lambda^2}+B\overline{\lambda} +A)$ and applying similar argument. So the inverse exists and we have the bound.
\end{proof}
Taking real part of the identity (\ref{eq1}) and applying similar argument, we can also prove the following lemma.
\begin{lemma}\label{lem3.2}
If ${\rm Re}\lambda>|{\rm Im}\lambda|$, then $R(\lambda)$ exists and the following bound holds
\begin{equation}\label{eq3}
\|R(\lambda)\|_{H\to H} ={\cal O}(|{\rm Re}\lambda|^{-1}).
\end{equation}
\end{lemma}
\begin{remark}
From the identity $(\lambda^2+B\lambda +A)^*=\overline{\lambda^2}+B\overline{\lambda} +A$, we have $R(\lambda)^*=R(\overline{\lambda}).$ So if $R(\lambda)$ exists then $R(\overline{\lambda})$ exists.
\end{remark}
Next we estimate $R(\lambda)$ as an operator between different Hilbert spaces.
\begin{lemma}\label{lem3.3}
If $R(\lambda)$ exists, then the following bound holds
\begin{equation}\label{eq4}
\|R(\lambda)\|_{ H\to H_1}={\cal O}(\|R(\lambda)\|_{H\to H}^{1/2}+(|\lambda|^{1/2}+|\lambda|)\|R(\lambda)\|_{H\to H}).
\end{equation}
\end{lemma}
\begin{proof}
For $u\in H$, we have 
\begin{align*}
\|R(\lambda)u\|^2_{H_1}&=(\sqrt A R(\lambda)u,\sqrt A R(\lambda) u)_H\\
&=(AR(\lambda)u,R(\lambda)u)_H\\
&=((A+\lambda^2+B\lambda)R(\lambda)u,R(\lambda)u)_H-((\lambda^2+B\lambda)R(\lambda)u,R(\lambda)u)_H\\
&=(u,R(\lambda)u)_H-((\lambda^2+B\lambda)R(\lambda)u,R(\lambda)u)_H\\
&={\cal O}(\|R(\lambda)\|_{H\to H}+(|\lambda|+|\lambda|^2)\|R(\lambda)\|_{H\to H}^2)\|u\|_H^2.
\end{align*}
\end{proof}
\begin{lemma}\label{lem3.4}
If $R(\lambda)$ exists, then $R(\lambda)A: D(A)\to H$ can be extend to $H_1\rightarrow H$ and the following bound holds
\begin{equation}\label{eq5}
\|R(\lambda)A\|_{ H_1\to H}={\cal O}(\|R(\lambda))\|_{H\to H}^{1/2}+(|\lambda|^{1/2}+|\lambda|)\|R({\lambda})\|_{H\to H}).
\end{equation}
\end{lemma}
\begin{proof}
By $R(\lambda)^*=R(\overline{\lambda})$ and $\rm{Lemma\ \ref{lem3.3}}$, $R(\overline \lambda)$ is a bounded operator from $H$ to $H_1$. Since $R(\lambda)^*=R(\overline{\lambda})$, we can regard $R(\lambda)$ as an operator from $H^{-1}$ to $H$ by duality argument. From $\rm{Lemma\  \ref{lem3.3}}$, we have the following estimate
$$R(\lambda): H^{-1}\rightarrow H={\cal O}(\|R({\lambda})\|_{H\to H}^{1/2}+(|\lambda|^{1/2}+|\lambda|)\|R({\lambda})\|_{H\to H}).$$
We extend $A$ to the operator form $H_1$ to $H^{-1}$ by using the following identity  
$$\langle Au,v\rangle_{H^{-1},H_1}=(\sqrt A u, \sqrt A v)_H.$$
From the definition, $A$ satisfies
$$A: H_1\rightarrow H^{-1}={\cal O}(1).$$
Thus we obtain the lemma. 
\end{proof}

\begin{lemma}\label{lem3.5}
If $R(\lambda)$ exists for a $\lambda\not=0$ then $R(\lambda)(\lambda+B): D(A)\to H_1$ can be extend to $H_1\to H_1$  and the following bound holds
\begin{equation}\label{eq6}
\|R(\lambda)(\lambda+B)\|_{H_1\to H_1}={\cal O}(( |\lambda|^{-1/2}+1)\|R({\lambda})\|_{H\to H}^{1/2}+(1+|\lambda|)\|R({\lambda})\|_{H\to H}+|\lambda|^{-1}).
\end{equation}
\end{lemma}
\begin{proof}
We will estimate $R(\lambda)(\lambda+B)$ using the following identity
$$R(\lambda)(\lambda+B)=\frac{1}{\lambda}-\frac{1}{\lambda}R(\lambda)A.$$
For $f\in D(A)$, we consider the following equation
$$R(\lambda)Af=u.$$
This can be written as follows
$$Af=(\lambda^2+B\lambda +A)u.$$
So we have
$$(Au,u)_H=(Af,u)_H-((\lambda^2+B\lambda)u,u)_H.$$
Taking real part, for sufficiently small $\varepsilon>0$, we obtain
$$\|u\|^2_{H_1}\leq C((|\lambda|^2+|\lambda|)\|u\|^2_H +\|f\|^2_{H_1})+\varepsilon\|u\|^2_{H_1}.$$
Thus 
$$\|R(\lambda)Af\|_{H_1}\leq C((|\lambda|^{1/2}_H+|\lambda|)\|R(\lambda)Af\|_{H} +\|f\|_{H_1}).$$
Applying $(\ref{eq5})$, the following bound holds
$$\|R(\lambda)Af\|_{H_1}\leq C((|\lambda|^{1/2}+|\lambda|)\|R({\lambda})\|_{H\to H}^{1/2}+(|\lambda|+|\lambda|^2)\|R({\lambda})\|_{H\to H})+1)\|f\|_{H_1}).$$
By density, we conclude the extension and the estimate.
\end{proof}
By these lemmas, we can give the existence of the resolvent.
\begin{proposition}\label{resolvent(prop)}
If $R(\lambda)$ exists for a $\lambda\not=0$, then the resolvent $(\lambda-{\cal A})^{-1}:\ {\cal H}_0\rightarrow {\cal H}_0$ exists and the following identity holds
\begin{equation}\label{resolvent}
(\lambda-{\cal A})^{-1}=
\begin{pmatrix}
R(\lambda)(\lambda+B) & R(\lambda) \\
-R(\lambda)A& R(\lambda)\lambda
\end{pmatrix}.
\end{equation}
The following estimate also holds 
\begin{equation}
\|(\lambda-{\cal A})^{-1}\|_{{\cal H}_0\to{\cal H}_0}={\cal O}((1+|\lambda|^{-1/2})\|R(\lambda)\|^{1/2}_{H\to H}+(1+|\lambda|)\|R(\lambda)\|_{H\to H})+|\lambda|^{-1}).
\end{equation}
\end{proposition}
\begin{proof}
From $\rm{Lemma\ \ref{lem3.3}}$, $\rm{Lemma\ \ref{lem3.4}}$ and $\rm{Lemma\ \ref{lem3.5}}$, the right hand side of ($\ref{resolvent}$) can be extend as an bounded operator on ${\cal H}_0$. By the direct computation, if ${}^t(u,v)\in D(A)\times D(A)$ then
$$\begin{pmatrix}
R(\lambda)(\lambda+B) & R(\lambda) \\
-R(\lambda)A& R(\lambda)\lambda
\end{pmatrix}
(\lambda-{\cal A})
\begin{pmatrix}
u \\
v
\end{pmatrix}
=
\begin{pmatrix}
R(\lambda)(\lambda+B) & R(\lambda) \\
-R(\lambda)A& R(\lambda)\lambda
\end{pmatrix}
\begin{pmatrix}
\lambda & -Id \\
A& \lambda+B
\end{pmatrix}
\begin{pmatrix}
u \\
v
\end{pmatrix}
=
\begin{pmatrix}
u \\
v
\end{pmatrix}$$
So by density it is the left inverse. On the other hand for ${}^t(u,v)\in D(A)\times D(\sqrt A)$
$$
\begin{pmatrix}
\lambda & -Id \\
A& \lambda+B
\end{pmatrix}
\begin{pmatrix}
R(\lambda)(\lambda+B) & R(\lambda) \\
-R(\lambda)A& R(\lambda)\lambda
\end{pmatrix}
\begin{pmatrix}
u \\
v
\end{pmatrix}
=\begin{pmatrix}
Id & 0 \\
AR(\lambda)(\lambda+B)-(\lambda+B)R(\lambda)A& Id
\end{pmatrix}
\begin{pmatrix}
u \\
v
\end{pmatrix}$$
We use the following trick, for a $\lambda\not=0$
$$\lambda(AR(\lambda)(\lambda+B)-(\lambda+B)R(\lambda)A)u=(AR(\lambda)(\lambda^2+B\lambda+ A)-(\lambda^2+B\lambda+A)R(\lambda)A)u=0.$$
Thus the density argument, it is also the right inverse. From $\rm{Lemma\ \ref{lem3.3}}$, $\rm{Lemma\ \ref{lem3.4}}$ and $\rm{Lemma\ \ref{lem3.5}}$, we have the bound.
\end{proof}
\begin{remark}
From this proposition, we have the following identities
\begin{align*}
((\lambda-{\cal A})^{-1}\ {}^t(0,u),{}^t(0,v))_{{\cal H}_0}=\lambda(R(\lambda)u,v)_H,\\
((\lambda-{\cal A})^{-1}\ {}^t(0,u),{}^t(v,0))_{{\cal H}_0}=(R(\lambda)u,v)_{H_1}.
\end{align*}
The left hand side is complex analytic since the resolvent is analytic. So the right had side is complex analytic so $R(\lambda): H\to H, H\to H_1$ are complex analytic family of  operators.
\end{remark}
Applying $\rm{Lemma\  \ref{lem3.3}}$ directly to $R(\lambda)(\lambda+B):\ H \rightarrow H_1$, we also have the following estimates to the resolvent. 
\begin{proposition}
If $R(\lambda)$ exists for a $\lambda\not=0$, then we have the following estimate
\begin{equation}
\|(\lambda-{\cal A})^{-1}\|_{{\cal H}\rightarrow {\cal H}_0}={\cal O}((|\lambda|+1)\|R(\lambda)\|_{H\to H}^{1/2}+(|\lambda|^{1/2}+|\lambda|^2)\|R(\lambda)\|_{H\to H}).
\end{equation}
\end{proposition}
This proposition says if we restrict the domain of the resolvent to $\cal H$ then the singularity, as $\lambda$ tend to $0$, become weaker one. This estimate is important for energy decay. We also have the following estimate. 
\begin{proposition}
If $R(\lambda)$ exists for a $\lambda\not=0$, then we have the following estimate
\begin{equation}
\|(\lambda-{\cal A})^{-1}\|_{{\cal H}\to {\cal H}}={\cal O}((1+|\lambda|^{-1/2})\|R(\lambda)\|_{H\to H}^{1/2}+(1+|\lambda|)\|R(\lambda)\|_{H\to H}+|\lambda|^{-1}).
\end{equation}
\end{proposition}
By $\rm{Lemma\ \ref{lem3.1}}$ and $\rm{Lemma\ \ref{lem3.2}}$, we conclude the following estimates of the resolvent.
\begin{proposition}\label{resolvent estimate}
If $|{\rm Re}\lambda|<c$ for sufficient small $c>0$ and ${\rm Im}\lambda\not=0$, then the resolvent satisfies
\begin{align*}
&\|(\lambda-{\cal A})^{-1}\|_{{\cal H}_0\to {\cal H}_0}={\cal O}(|{\rm Im}\lambda|^{-1}+1),\\
&\|(\lambda-{\cal A})^{-1}\|_{{\cal H}\to {\cal H}_0}={\cal O}((1+|\lambda |)|{\rm Im}\lambda|^{-1/2}+(|\lambda|^{1/2}+|\lambda|^2)|{\rm Im}\lambda|^{-1}),\\
&\|(\lambda-{\cal A})^{-1}\|_{{\cal H}\to{\cal H}}={\cal O}(|{\rm Im} \lambda|^{-1}+1).
\end{align*}
If ${\rm Re}\lambda>0$, the following estimate also holds
\begin{align*}
&\|(\lambda-{\cal A})^{-1}\|_{{\cal H}_0\to {\cal H}_0}={\cal O}(|\lambda|^{-1}+1),\\
&\|(\lambda-{\cal A})^{-1}\|_{{\cal H}\to {\cal H}_0}={\cal O}(1+|\lambda|+|\lambda|^{-1/2}),\\
&\|(\lambda-{\cal A})^{-1}\|_{{\cal H}\to {\cal H}}={\cal O}(|\lambda|^{-1}+1).
\end{align*}

\end{proposition}
\end{section}
\begin{section}{Abstract energy decay} 
We prove the following  abstract energy decay estimate to the damped wave equation. The theorem may be well known but we prove here for later argument.  
\begin{theorem}
There exists a constant $C$ such that for any ${}^t(u_0,u_1)\in {\cal H}$, and $t>1$, the following bound holds
$$\|U(t){}^t(u_0,u_1)\|_{{\cal H}_0}\leq C(t^{-1/2}\|{}^t(u_0,u_1)\|_{{\cal H}}).$$

\end{theorem}
\begin{remark} 
We give one example which may be useful to understand the proof of this theorem. We consider the constant coefficient damped wave equation on $\R^d$
$$(\pl_t^2 -\Delta +\pl_t )u=0.$$
The Fourier transform of the solution is the following form
$$u_+(\xi){\rm exp}\left(\frac{-1+ \sqrt{1-4|\xi|^2}}{2}t\right)+u_-(\xi)\exp\left(\frac{-1- \sqrt{1-4|\xi|^2}}{2}t\right). $$
For simplicity, we assume $u_+$ and $u_-$ are in $L^2$, then $u_-$ part decays exponentially so it is no interest. If $\xi=0$, $u_+$ part does not decay therefore uniform decay does not occur if we only assume $u_+(\xi) \in L^2$. 
However if we assume $|\xi|u_+(\xi) \in L^2$, we have the uniform decay estimate, by using the following decomposition,
\begin{align*}
&\left \||\xi|u_+(\xi){\rm exp}(\frac{-1+ \sqrt{1-4|\xi|^2}}{2}t)\right\|_{L^2}\\
&\leq\left\|\chi_{\{|\xi|<1/2\}}|\xi|u_+(\xi){\rm exp}(\frac{-1+ \sqrt{1-4|\xi|^2}}{2}t)\right\|_{L^2}+\left\|\chi_{\{|\xi|\geq 1/2\}}|\xi|u_+(\xi){\rm exp}(\frac{-1+ \sqrt{1-4|\xi|^2}}{2}t)\right\|_{L^2}\\
&\leq\left\|\chi_{\{ |\xi|<1/2 \}}|\xi|{\rm exp}(-|\xi|^2t) \right\|_{L^\8} \|u_+(\xi)\|_{L^2}+{\cal O}(e^{-t})\||\xi|u_+(\xi)\|_{L^2}\\
&\leq {\cal O}(t^{-1/2})(\|u_+(\xi)\|_{L^2}+\||\xi|u_+(\xi)\|_{L^2}).
\end{align*}
Here $\chi_{\{ \cdot \}}$ denote the characteristic functions. 
The high frequency part decays exponentially and the low frequency part becomes the main part. 
This type decomposition will be used in the proof of the theorem. We also remark that the above estimate is optimal without any further assumption for $u_+$.  
\end{remark}
We give a proof of the theorem. First we assume the Cauchy data $f={}^t(u_0,u_1)$ is in $D({\cal A}^k) $ for large $k$ and $u_0\in H$.  So $U(t)f$ is in $C^k$ as $t$ variable function. We take a smooth cut-off function $\psi$ satisfying
$$\psi(t)=
\begin{cases}

1, \ t\geq2,\\
0,\ t<1.
\end{cases}
$$
We shall estimate the following cut-off propagator as in $\cite{Bu},$ $\cite{Kh}$, $\cite{Zw}$.
$$V(t)f=\psi(t)U(t)f.$$
From $(\ref{propa})$, we can define its Fourier-Laplace transform by 
$$\widehat{ Vf}(\lambda)=\frac{1}{\sqrt{2\pi}}\int^\8_{-\8} e^{-i\lambda t}V(t)fdt$$ 
if ${\rm Im}\ \lambda=-{\rm Re}\ i \lambda=-\varepsilon<0$.
Then by Fourier inversion formula, the following identity holds 
$$V(t)f=\frac{1}{\sqrt{2\pi}}\int_{{\rm Im}\lambda=-\varepsilon} e^{i\lambda t}{\widehat{Vf}(\lambda)}d\lambda.$$
By definition, we have the identity
$$(\pl_t-{\cal A})V(t)f=\psi'(t)U(t)f.$$
Taking its Fourier-Laplace transform, we get
$$(i\lambda-{\cal A}){\widehat{Vf}(\lambda)}=\widehat{\psi'Uf}(\lambda).$$
if ${\rm Im}\ \lambda=-\varepsilon<0$. Since $\cal A$ is a dissipative operator the resolvent exists for ${\rm Im}\lambda=-\varepsilon<0$ and we obtain 
$$\widehat{ Vf}(\lambda)=(i\lambda-{\cal A})^{-1}\widehat{\psi'Uf}(\lambda).$$
We take its inverse Fourier transform and we get
$${V(t)}f=\frac{1}{\sqrt{2\pi}}\int_{{\rm Im}\lambda=-\varepsilon}(i\lambda-{\cal A})^{-1} e^{i\lambda t}\widehat{\psi'Uf}(\lambda)d\lambda.$$
Since $\psi'(t)$ has a compact support, $\widehat{\psi'Uf}(\lambda)$ is a holomorphic function and by the results of previous section, $(i\lambda-{\cal A})^{-1}$ can be holomorphically extended to ${\rm Re}\  i \lambda=-a<0$ if $a>0$ is a sufficiently small constant and  ${\rm Im}\ i \lambda \not=0$. Since $U(t)f$ is $C^k$ in $D(A)\times H$, $\widehat{\psi'Uf}(\lambda)$ decays sufficiently first as $\lambda$ tend to infinity, so we can change the integral contour as follows
\begin{align*}
{V(t)}f&=\frac{1}{\sqrt{2\pi}}\int_{{\rm Im}\lambda=
a}\chi(\lambda)(i\lambda-{\cal A})^{-1} e^{i\lambda t}\widehat{\psi'Uf}(\lambda)d\lambda+\frac{1}{\sqrt{2\pi}}\int_{\cal C}(i\lambda-{\cal A})^{-1} e^{i\lambda t}\widehat{\psi'Uf}(\lambda)d\lambda\\
&=e^{-at}I_1f(t)+I_2f(t)=\text{(High frequency part) $+$ (low frequency part)}.
\end{align*}
Here $\chi$ is a cut-off function defined by 
$$\chi(\lambda)=
\begin{cases}
1, \ |{\rm Re}\lambda|\geq\delta,\\
0,\ |{\rm Re}\lambda|<\delta,
\end{cases}
$$
for sufficient small $\delta>0$. $\cal C$ which we choose later, is a contour around the origin connecting two points $\pm\delta+i a$ . We prove the following bound of $I_1$ by using the method of $\cite{Kh}$ which is known as Morawetz's argument
\begin{equation}\label{exp decay}
\|e^{-at}I_1f(t)\|_{{\cal H}_0}\leq C e^{(-a+\varepsilon)t}\|f\|_{{\cal H}_0}. 
\end{equation}
First we estimate $L^2$ norm of $I_1f$. Using Plancheral's identity, we have
\begin{align*}
\int^{\8}_{-\8}\|I_1f(\tau)\|_{{\cal H}_0}^2d\tau&=\int^{\8}_{-\8}\left\|\frac{1}{\sqrt{2\pi}}\int_{{\rm Im}\lambda=
a}e^{(i\lambda+a) \tau}\chi(\lambda)(i\lambda-{\cal A})^{-1} \widehat{\psi'Uf}(\lambda)d\lambda\right\|_{{\cal H}_0}^2d\tau\\
&\leq\int^{\8}_{-\8}\|\chi(\lambda+ia)(i\lambda-a-{\cal A})^{-1} \widehat{\psi'Uf}(\lambda+ia)\|_{{\cal H}_0}^2d\lambda.
\end{align*}
By $\rm{proposition\ \ref{resolvent estimate}}$ and $\chi$ cut low frequency, $\chi(\lambda+ia)(i\lambda-a-{\cal A})^{-1}$ are uniformly bounded operators on ${\cal H}_0$ so we estimate
\begin{align*}
\int^{\8}_{-\8}\|I_1f(\tau)\|_{{\cal H}_0}^2d\tau&\leq C\int^{\8}_{-\8}\left\|\widehat{\psi'Uf}(\lambda+ia)\right\|_{{\cal H}_0}^2d\lambda\\
&=C \int^{+\8}_{-\8}e^{2at}\|\psi'(t)U(t)f\|_{{\cal H}_0}^2dt.
\end{align*}
For last identity, we use Plancheral's identity again. Since $\psi'$ is a compact support function, we have the following estimate
\begin{equation}\label{I_1-L^2est}
\int^{\8}_{-\8}\|I_1f(\tau)\|_{{\cal H}_0}^2d\tau\leq C\|f\|^2_{{\cal H}_0}.
\end{equation}
We give point-wise bound of $I_1f(t)$ using this estimate. 
From this inequality, there exist $0<s<2$ such that 
$$\|I_1f(s)\|_{{\cal H}_0}\leq C\|f\|_{{\cal H}_0}.$$
We define
\begin{align*}
\pl_t I_1 f(t)-{\cal A}I_1 f(t)&=\frac{1}{\sqrt{2\pi}}\int_{{\rm Im}\lambda=
a}\chi(\lambda) e^{(i\lambda+a) t}\widehat{\psi'Uf}(\lambda)d\lambda+aI_1f(t)\\
&=I_3f(t)+aI_1f(t).
\end{align*}
By Duhamel's principle, we have
\begin{equation}\label{du}
I_1f(t)=U(t-s)I_1f(s)+\int_{s}^tU(t-\tau)(I_3f(\tau)+aI_1f(\tau))d\tau.
\end{equation}
We shall estimate the last part. First we obtain
\begin{equation}\label{small change}
\begin{split}
\int^{t}_{s}\|U(t-\tau)I_3f(\tau)\|_{{\cal H}_0}^2d\tau
&=\int^{t}_{s}\left\|U(t-\tau)\int_{{\rm Im}\lambda=
a} e^{(i\lambda+a) \tau}\chi(\lambda)\widehat{\psi'Uf}(\lambda))d\lambda\right\|^2_{{\cal H}_0}d\tau\\
&\leq  \int^{+\8}_{-\8}\left\|\int_{{\rm Im}\lambda=
a} e^{(i\lambda+a) \tau}\chi(\lambda)\widehat{\psi'Uf}(\lambda))d\lambda\right\|^2_{{\cal H}_0}d\tau
\end{split}
\end{equation}
since $U(t)$ is a propagator of a dissipative operator. By Plancheral's identity we have
\begin{align*} 
\int^{t}_{s}\|U(t-\tau)I_3f(\tau)\|_{{\cal H}_0}^2d\tau&\leq  \int^{+\8}_{-\8}\| \chi(\lambda+ia)\widehat{\psi'Uf}(\lambda+ia)\|_{{\cal H}_0}^2d\lambda\\
&\leq C \int^{+\8}_{-\8}\|\widehat{\psi'Uf}(\lambda+ia)\|_{{\cal H}_0}^2d\lambda.
\end{align*}
We use Plancheral's identity again and since $\psi'$ is a compact support function, we have
\begin{equation}\label{I_3-L^2est}
\begin{split}
\int^{t}_{s}\|U(t-\tau)I_3f(\tau)\|_{{\cal H}_0}^2d\tau&\leq C \int^{+\8}_{-\8}e^{2at}\|\psi'(t)U(t)f\|_{{\cal H}_0}^2dt\\
&\leq C\|f\|^2_{{\cal H}_0}.
\end{split}
\end{equation}
By the Schwartz inequality, from $(\ref{du})$, $(\ref{I_1-L^2est})$ and $(\ref{I_3-L^2est})$, we have
\begin{align*}
e^{-\varepsilon t }&\|I_1f(t)\|_{{\cal H}_0}\leq e^{-\varepsilon t } \|U(t-s)I_1f(s)\|_{{\cal H}_0}+\left\| \int_{s}^te^{-\varepsilon t }U(t-\tau)( I_3f(\tau)+aI_1f(\tau))d\tau\right\|_{{\cal H}_0}\\
&\leq C\|f\|_{{\cal H}_0}+\left\{ \int_{s}^te^{-2\varepsilon t}d\tau\right\}^{1/2}\left\{ \int_{s}^t\|U(t-\tau)I_3f(\tau)\|_{{\cal H}_0}^2d\tau+\int_{s}^t\|U(t-\tau)aI_1f(\tau)\|_{{\cal H}_0}^2d\tau\right\}^{1/2}\\
&\leq C\|f\|_{{\cal H}_0}+C\left\{ \int_{s}^t\|U(t-\tau)I_3f(\tau)\|_{{\cal H}_0}^2d\tau +\int_{-\8}^\8\|I_1f(\tau)\|_{{\cal H}_0}^2d\tau\right\}^{1/2}\\
&\leq C \|f\|_{{\cal H}_0}.
\end{align*}
So $e^{-at}I_1 f(t)$ decays exponentially.

Next we study $I_2f$ part. We show the following estimate.
\begin{proposition}\label{low decay}
There exists $C>0$ such that for any $t>1$ and $f\in {\cal H}$, the following bound holds
$$\|I_2f(t)\|_{{\cal H}_0}\leq C(t^{-1/2}\|f\|_{{\cal H}}).$$
\end{proposition}
\begin{proof}
We choose the contour as follows ${\cal C}={\cal C}_o \cup  {\cal C}_+\cup  {\cal C}_-$. ${\cal C}_o=\{\frac{1}{t}e^{is}; s\in[-\pi,0] \}$, ${\cal C}_+=\{(1-s)\frac{1}{t}+s(\delta+ia); s\in[0,1] \}$, ${\cal C}_-=\{-(1-s)\frac{1}{t}+s(-\delta+ia); s\in[0,1] \}$. Here we impose suitable orientation on these contours. We estimate each contours in the following integral,
$$I_2f(t)=\frac{1}{\sqrt{2\pi}}\int_{{\cal C}_o \cup  {\cal C}_+\cup  {\cal C}_-}(i\lambda-{\cal A})^{-1} e^{i\lambda t}\widehat{\psi'U(\lambda)}fd\lambda.$$
We have $\|(i\lambda-{\cal A})^{-1}\|_{{\cal H}\to {\cal H}_0}={\cal O}(t^{1/2})$ on ${\cal C}_o$ by Proposition $\ref{resolvent estimate}$ and the length of ${\cal C}_o$ is ${\cal O}(1/t)$. So the following bound holds
\begin{align*}
\left\|\int_{{\cal C}_o}(i\lambda-{\cal A})^{-1} e^{i\lambda t}\widehat{\psi'Uf}(\lambda)d\lambda\right\|_{{\cal H}_0}&={\cal O}\left( \frac{1}{t}t^{1/2}\sup_{\lambda\in {\cal C}_o}\|\widehat{\psi'Uf}(\lambda)\|_{\cal H}\right).
\end{align*}
From $(\ref{propa})$ and $\psi'$ is a compact support function, we have 
$$\sup_{\lambda\in {\cal C}_o}\|\widehat{\psi'Uf}(\lambda)\|_{\cal H}=\sup_{\lambda\in {\cal C}_o}\left\|\frac{1}{\sqrt{2\pi}}\int^\8_{-\8} e^{-i\lambda s}\psi'(s)U(s) fds\right\|_{\cal H} = {\cal O}(1)\|\psi'\|_{L^1}\|f\|_{\cal H}$$
So we have the estimate for ${\cal C}_o$ part. Next we estimate $\cal C_+$ part, for sufficient large $t$, we have $|\lambda|\sim |{\rm Re}\lambda|\sim s+1/t$ on ${\cal C}_+$. So by Proposition $\ref{resolvent estimate}$, we obtain
$$ \|(i\lambda-{\cal A})^{-1}\|_{{\cal H}\to {\cal H}_0} ={\cal O}(t^{1/2}).$$
Thus 
\begin{align*}
\left\|\int_{{\cal C}_+}(i\lambda-{\cal A})^{-1} e^{i\lambda t}\widehat{\psi'Uf}(\lambda)d\lambda\right\|_{{\cal H}_0}&={\cal O}\int_{0}^1t^{1/2} e^{-ast}ds \sup_{\lambda\in {\cal C}_+}\|\widehat{\psi'Uf}(\lambda)\|_{\cal H}\\
&={\cal O}(t^{-1/2}\|f\|_{\cal H}).
\end{align*}
In the same way, we can estimate ${\cal C}_-$ part and we have proved.
\end{proof}
Thus we have proved the abstract energy decay of the damped wave equation for sufficiently regular initial data. Then by density argument, we obtain the theorem. 
\begin{remark}\label{ene rem}
In the proof, we essentially use the following facts

$(i)$ the existence of the resolvent for ${\rm Im} \lambda\not=0$ and ${\rm Re}\lambda>-c$, $c>0$,

$(ii)$ uniform estimate of resolvent for large $\lambda$ which comes from $|Im \lambda|^{-1}$ order estimate of $R(\lambda),$

$(iii)$ $|\lambda|^{-1/2}$ order estimate of the resolvent near the origin which comes from $|\lambda|^{-1}$ estimate of $R(\lambda)$ for ${\rm Re} \lambda>0$ and ${\rm Im} \lambda \sim |\lambda|$ regions.
\end{remark}
\end{section}

\begin{section}{Proof of Theorem ${\ref{main}}$}
From now on we shall prove Theorem ${\ref{main}}$.
We recall $\tilde A=B^{-1/2}AB^{ -1/2}.$ 
Since $B^{-1/2}$ is bounded from $D(A)$ to $D(A)$, this operator is also a self-adjoint operator whose domain is $D(A)$. By the spectral theorem, we introduce a projection operator $\phi(\tilde A)$. 
Here $\phi$ is a cut-off function defined as follows
\begin{equation}\label{cutft}
\phi(x)=
\begin{cases}
1, \ x\geq\varepsilon,\\
0,\ x<\varepsilon.
\end{cases}
\end{equation}
By using this function, we define a cut-off operator by
$$\phi_B(A)=B^{-1/2}\phi(\tilde A)B^{1/2}. $$
Then we have
$$\phi_B(A)^2=\phi_B(A),\ \ \phi_B(A)^*=B^{1/2}\phi(\tilde A)B^{-1/2}.$$
Since $U(t): {\cal H}\rightarrow {\cal H}$, we can apply $\phi_B(A)$ to $U(t)f={}^t(u(t),v(t))$, $f={}^t(u_0,u_1)\in {\cal H}$ by
$$\phi_B(A)U(t)f={}^t(\phi_B(A)u(t),\phi_B(A)v(t)).$$
To prove Theorem ${\ref{main 2}}$, we use the following decomposition
\begin{align*}
U(t)f&=\phi_B(A)U(t)f+(1-\phi_B(A))U(t)f\\
&=(\text{High frequency part}) + (\text{Low frequency part})
\end{align*}
and estimate each part. 
First we estimate the High frequency part. 
\begin{proposition}\label{high}
There exists $C>0$ such that for any $t>1$ and $f\in {\cal H}$, we have the following bound
$$\|\phi_B(A)U(t)f\|_{\cal H}\leq Ct^{-2}\|f\|_{\cal H}.$$
\end{proposition}

\begin{proof}
Applying similar argument to the previous section, for sufficient regular $f$, we write 
\begin{align*}
\phi_B(A)\psi(t)U(t)f&=\phi_B(A)\frac{1}{\sqrt{2\pi}}\int_{{\rm Im}\lambda=
a}\chi(\lambda)(i\lambda-{\cal A})^{-1} e^{i\lambda t}\widehat{\psi'Uf}(\lambda)d\lambda\\
&\quad +\phi_B(A)\frac{1}{\sqrt{2\pi}}\int_{\cal C}(i\lambda-{\cal A})^{-1} e^{i\lambda t}\widehat{\psi'Uf}(\lambda)d\lambda\\
&=\phi_B(A)e^{-at}I_1f(t)+\phi_B(A)I_2f(t)\\
&=(\text{High frequency part})+(\text{Low frequency part})
\end{align*}
Here $\psi (t)$ and $\chi(\lambda)$ are cut-off functions introduced in previous section, $a>0$ is a sufficient small constant and ${\cal C}$ is the contour around the origin which is also used in the former section. Since $\psi(t)=1$ for large $t$, we have
$$\phi_B(A)U(t)f=\phi_B(A)\psi(t)U(t)f$$
for large $t$. So we need the estimate of $\phi_B(A)\psi(t)U(t)f$. The estimate of $e^{-at}I_1f(t)$ is essentially similar to the argument to Proposition $\ref{exp decay}$ but we need small modification in $(\ref{small change})$ since we only know the estimate $(\ref{propa})$ and $U(t-\tau)$ may be not bounded. 
In this case, we can prove the following estimate by similar argument for arbitrary small $\varepsilon>0$
$$\int^{t}_{s}\|e^{-\varepsilon t}U(t-\tau)I_3f\|_{{\cal H}}^2d\tau \leq C\|f\|_{\cal H}.$$
So we conclude the following bound
\begin{equation}\label{I_1}
\|\phi_B(A)e^{-at}I_1f(t)\|_{\cal H}={\cal O}(e^{-(a-2\varepsilon)t}\|f\|_{\cal H}).
\end{equation}
Next we shall estimate $\phi_B(A)I_2f$ by using the operator
$$(B\lambda+A)^{-1}=B^{-1/2}(\lambda+\tilde A)^{-1}B^{-1/2}.$$
This identity is easily seen from $(B\lambda+A)=B^{1/2}(\lambda+\tilde A)B^{1/2}$. 
From the following identity
$$(\lambda^2+B\lambda+A)(B\lambda+A)^{-1}=Id +\lambda^2(B\lambda+A)^{-1}, $$
we also have
\begin{equation}\label{asym resolvent}
R(\lambda)=(B\lambda+A)^{-1}-\lambda^2(B\lambda+A)^{-1}R(\lambda).
\end{equation}
Here $R(\lambda)=(\lambda^2+B\lambda+A)^{-1}$ and $(B\lambda+A)^{-1}$ are commutative operators. We recall the identity
\begin{align}\label{reso}
(\lambda-{\cal A})^{-1}=
R(\lambda)\begin{pmatrix}
(\lambda+B) & Id \\
-A& \lambda\\
\end{pmatrix}=R(\lambda)M(\lambda).
\end{align}
Applying these identities, we have
\begin{align*}
\phi_B(A)I_2f(t)&=
\frac{1}{\sqrt{2\pi}}\int_{\cal C}\phi_B(A)\{(Bi\lambda+A)^{-1}+\lambda^2(Bi\lambda+A)^{-1}R(i\lambda)\}M(i\lambda)e^{i\lambda t}\widehat{\psi'Uf}(\lambda)d\lambda\\
&=J_1f(t)+J_2f(t).
\end{align*}
Here
$$J_1f(t)=\frac{1}{\sqrt{2\pi}}\int_{\cal C}B^{-1/2}\phi(\tilde A)(i\lambda+\tilde A)^{-1} B^{-1/2}e^{i\lambda t}M(i\lambda)\widehat{\psi'Uf}(\lambda)d\lambda,$$
\begin{equation}
J_2f(t)=\frac{1}{\sqrt{2\pi}}\int_{\cal C}B^{-1/2}\phi(\tilde A)(i\lambda+\tilde A)^{-1} B^{-1/2}\lambda^2(i\lambda-{\cal A})^{-1}e^{i\lambda t}\widehat{\psi'Uf}(\lambda)d\lambda.
\end{equation}
Since $\phi(\tilde A)$ is a cut-off operator to the high frequency part. $\phi(\tilde A)(i\lambda+\tilde A)^{-1}$ can be holomorphically extended across the origin. Since $B^{1/2}$ and $B^{-1/2}$ are continuous on $D(A)$, by interpolation we have the following estimate if $\lambda$ is sufficiently near the origin, 
\begin{align*}
\|B^{-1/2}\phi(\tilde A)(i\lambda +\tilde A)^{-1}u\|^2_{H_1}&=\|{\sqrt A}B^{-1/2}\phi(\tilde A)(i\lambda +\tilde A)^{-1}u\|^2_{H}\\
&=({\tilde A}\phi(\tilde A)(i\lambda +\tilde A)^{-1}u,\phi(\tilde A)(i\lambda +\tilde A)^{-1}u)_H \\
&=\|{\sqrt{\tilde A}}\phi(\tilde A)(i\lambda +\tilde A)^{-1}u\|^2_{H}\\
&=\|\phi(\tilde A)(i\lambda +\tilde A)^{-1}{\sqrt{\tilde A}}u\|^2_{H}\\
&={\cal O}(\|{\sqrt{\tilde A}}u\|^2_{H})={\cal O}(\|u\|^2_{H_1}+\|u\|^2_H).
\end{align*}
In the same way, using  $(i \lambda+\tilde A)^{-1}\tilde A= Id -i\lambda(i \lambda+\tilde A)^{-1}$, near the origin, we have the following bound 
$$\|B^{-1/2}\phi(\tilde A)(i\lambda +\tilde A)^{-1}B^{-1/2}M(i\lambda)\|_{{\cal H}\to {\cal H}}={\cal O}(1).$$
For $J_1f$, we deform the contour to ${\rm Re}\  i\lambda<0$. Then we have
\begin{align*}
\|J_1f(t)\|_{\cal H}&=\frac{1}{\sqrt{2\pi}}\left\|\int_{\rm {Im}  \lambda =a, |{\rm Re}  \lambda|<\delta}B^{-1/2}\phi(\tilde A)(i\lambda+\tilde A)^{-1} B^{-1/2}e^{i\lambda t}M(i\lambda)\widehat{\psi'Uf}(\lambda)d\lambda\right\|_{\cal H}\\
&\leq C \int_{\rm {Im}\lambda =a, |{\rm Re}\lambda|<\delta} |e^{i\lambda t}|\|\widehat{\psi'Uf}(\lambda)\|_{\cal H}d\lambda\\
&\leq C e^{-a t}\|f\|_{\cal H}.
\end{align*}
Thus this part decays exponentially. 
We apply same contour deformation as in Proposition $\ref{low decay}$ to $J_2f$. Since $\|\lambda^2(i\lambda+{\cal A})^{-1}\|_{{\cal H}\to {\cal H}}={\cal O} (\lambda)$ for ${\rm Re}\  i \lambda>0$, we obtain 
$\|\lambda^2(i\lambda+{\cal A})^{-1}\|_{{\cal H}\to {\cal H}}={\cal O} (1/t)$ on ${\cal C}_o$ which is defined in the proof of Proposition $\ref{low decay}$. Then the length of ${\cal C}_o$ is ${\cal O} (1/t)$ and $\phi(\tilde A)(i\lambda+\tilde A)^{-1}$ is bounded, we obtain
$$\left\|\int_{{\cal C}_o}B^{-1/2}\phi(\tilde A)(i\lambda+\tilde A)^{-1} B^{-1/2}\lambda^2(i\lambda+{\cal A})^{-1}e^{i\lambda t}\widehat{\psi'Uf}(\lambda)d\lambda\right\|_{\cal H}={\cal O}(t^{-2}\|f\|_{\cal H}).$$
On ${\cal C}_+$, we have $|\lambda| \sim |{\rm Re}\lambda|\sim s+1/t$. So $\|(i\lambda+{\cal A})^{-1}\|_{{\cal H}\to {\cal H}}={\cal O}(|\lambda|^{-1})={\cal O}(t)$ and 
\begin{align*}
&\left\|\int_{{\cal C}_+}B^{-1/2}\phi(\tilde A)(i\lambda+\tilde A)^{-1} B^{-1/2}\lambda^2(i\lambda+{\cal A})^{-1}e^{i\lambda t}\widehat{\psi'Uf}(\lambda)d\lambda\right\|_{\cal H}\\
&\quad ={\cal O}(1)\int_{0}^1t (s+1/t)^2 e^{-ast}ds\|f\|_{\cal H}\\
&\quad ={\cal O}(t^{-2}\|f\|_{\cal H}).
\end{align*}
We can get the same estimate for ${\cal C}_-$. Now by density, we have proved.
\end{proof}
Next we estimate the low frequency part: $(1-\phi_B(A))U(t)f=\tilde\phi_B(A)U(t)f$. Here we write
$$\tilde \phi=1-\phi.$$
We assume $f$ is sufficiently regular then by the similar argument used in the previous section, for a cut-off function $\psi$, the following identity holds
$$\tilde \phi_B(A)\psi(t)U(t)f=\frac{1}{\sqrt{2\pi}}\int_{{\rm Im}\lambda=
-\varepsilon}\tilde \phi_B(A)R(i\lambda)M(i\lambda)e^{i\lambda t}\widehat{\psi'Uf}(\lambda)d\lambda.$$
Writing $U(t)f={}^t(u(t),v(t))$, $\widehat{\psi'Uf}(\lambda)={}^t({\widehat{\psi'Uf}(\lambda)}_0,{\widehat{\psi'Uf}(\lambda)}_1)$, we have the following identity for large $t$
\begin{equation}\label{high eq}
\tilde \phi_B(A)u(t)=\frac{1}{\sqrt{2\pi}}\int_{{\rm Im}\lambda=
-\varepsilon}\tilde \phi_B(A)R(i\lambda)e^{i\lambda t}\left((i\lambda+B){\widehat{\psi'Uf}(\lambda)}_0+{\widehat{\psi'Uf}(\lambda)}_1\right)d\lambda.
\end{equation}
We define $\psi_n(t)=\psi(nt)$ for the cut-off function $\psi$ satisfying $\psi'\geqq0$. Then $\psi_n'$ become a Dirac sequence i.e. $\psi_n'(t)\to\delta(t)$, as $n\to\8$ where $\delta(t)$ is the Dirac measure. We change $\psi$ in the above integral to a sequence of cut-off functions $\{\psi_n\}$. 
So $\widehat{\psi_n'Uf}{(\lambda)}\to {\frac{1}{\sqrt 2\pi}} U(0)f=\frac{1}{\sqrt{2\pi}}f $ as $n\to\8$ and we would like to take this limit in the integral. 
To verify this limiting argument, we use the following lemma.
\begin{lemma}\label{cutoff resolvent}
If $R(\lambda)$ exists for a $\lambda\not=0$ then we have the following bound
$$\|\tilde\phi_B(A)R(\lambda)u\|_H={\cal O}(|\lambda|^{-2}\|u\|_H)+{\cal O}(|\lambda|^{-2}(|\lambda|+1)\|R(\lambda)u\|_H),$$
$$\|\tilde\phi_B(A)R(\lambda)Av\|_H={\cal O}(|\lambda|^{-2}\|v\|_H)+{\cal O}(|\lambda|^{-2}(|\lambda|+1)\|R(\lambda)Av\|_H),$$
for $u \in H$ and $v\in D(A)$.
\end{lemma}
\begin{remark}
If $A$ is bounded, then $\lambda^2 +\lambda B+A={\cal O}(\lambda^2)$ for large $\lambda$. So we can expect nearly $|\lambda|^{-2}$ estimate of the high frequency cut-off resolvent.
\end{remark}
\begin{proof}
Since $\tilde\phi(\tilde A): H\to D(A)$ and $B^{1/2}: D(A)\to D(A)$ are bounded, we have
\begin{align*}
\|\tilde\phi_B(A) Au\|_H&=\|\tilde\phi_B(A) A\|_{H\to H}\| u\|_H=\|(\tilde\phi_B(A)A)^*\|_{H\to H}\| u\|_H\\ 
&=\|AB^{1/2}\tilde\phi(\tilde A)B^{-1/2}\|_{H\to H}\|u\|_H\leq C\|u\|_H.
\end{align*}
So we obtain 
$$\|(\tilde\phi_B(A)(\lambda^2+B\lambda+A)u\|_H\geq |\lambda|^2 \|\tilde\phi_B(A) u\|_H-C(|\lambda|+1)\|u\|_H.$$
Thus 
$$\|\tilde\phi_B(A) u\|_H\leq |\lambda|^{-2}\|(\tilde\phi_B(A)(\lambda^2+B\lambda+A)u\|_H+C|\lambda|^{-2}(|\lambda|+1)\|u\|_H.$$
Taking $u=R(\lambda)v$ and $u=R(\lambda)Av$, we have the desired inequality.
\end{proof}
  
Since $R(\lambda)(\lambda+B)=(1+ R(\lambda)A)/\lambda$, by $(\ref{high eq})$, we obtain
\begin{align*}
\tilde \phi_B(A)u(t)=&\frac{1}{\sqrt{2\pi}}\int_{{\rm Im}\lambda=
-\varepsilon}\tilde \phi_B(A)e^{i\lambda t}\frac{1}{i\lambda}{\widehat{\psi'Uf}(\lambda)}_0d\lambda\\
&+\frac{1}{\sqrt{2\pi}}\int_{{\rm Im}\lambda=
-\varepsilon}\tilde \phi_B(A)e^{i\lambda t}\left(\frac{1}{i\lambda}R(i\lambda)A{\widehat{\psi'Uf}(\lambda)}_0+R(i\lambda){\widehat{\psi'Uf}(\lambda)}_1\right)d\lambda\\
=&\tilde J_1(t)+\tilde J_2(t).
\end{align*} 
If $\supp\ \psi'\subset[0,c]$ we have the following estimate
$$\|{\widehat{\psi'Uf}(\lambda)}\|_{\cal H}=\left\|\frac{1}{\sqrt {2\pi}}\int^\8_{-\8} e^{-i\lambda s}\psi'(s)U(s)fds\right\|_{\cal H}\leq C e^{c{\rm Im }\lambda}\|\psi'\|_{L^1}\|f\|_{\cal H}=Ce^{c{\rm Im }\lambda}\|f\|_{\cal H}.$$
So for sufficient large $t$, in the integral $\tilde J_1$, we change the contour to $\tilde{\cal C}=\tilde {\cal C}_+ \cup \tilde {\cal C}_- \cup {\cal C}$. Here  
$\tilde {\cal C}_\pm=\{\pm s+is; s\in[\delta,\8)\}$ and ${\cal C}$ is a contour around the origin.
Thanks to the estimate $|e^{i\lambda t}|\leqq Ce^{-t{\rm Im }\lambda}$ on $\tilde{\cal C}_+ \cup \tilde{\cal C}_-$, if $t$ is sufficiently large, we have
\begin{align*}
\tilde J_1(t)&=\frac{1}{\sqrt{2\pi}}\int_{\tilde{\cal C}_+\cup\tilde{\cal C}_-\cup{\cal C}}\tilde \phi_B(A)e^{i\lambda t}\frac{1}{i\lambda}{\widehat{\psi'Uf}(\lambda)}_0d\lambda\\
&=\frac{1}{\sqrt{2\pi}}\int_{{\cal C}}\tilde \phi_B(A)e^{i\lambda t}\frac{1}{i\lambda}{\widehat{\psi'Uf}(\lambda)}_0d\lambda+{\cal O}(e^{-\delta(t-c)}\|f\|_{\cal H})
\end{align*}
 in ${ H}$. For $\tilde J_2$, we change the contour to $\widehat {\cal C}=\widehat{\cal C}_+ \cup \widehat{\cal C}_-  \cup {\cal C}$.  Here $\widehat{\cal C}_\pm=\{\pm s+i\varepsilon; s\in[\delta,\8)\}$ and ${\cal C}$ is a contour around the origin as in the proof of Proposition $\ref{low decay}$. By Lemma $\ref{cutoff resolvent}$, Lemma $\ref{lem3.4}$ and Lemma $\ref{lem3.1}$, on $\widehat{\cal C}_\pm$, we have the following estimate, 
$$\left\|\frac{1}{i\lambda}\tilde\phi_B(A)R(i\lambda)A{\widehat{\psi'Uf}(\lambda)}_0+\tilde\phi_B(A)R(i\lambda){\widehat{\psi'Uf}(\lambda)}_1\right\|_{H}={\cal O} (|\lambda|^{-2}\|{\widehat{\psi'Uf}(\lambda)}\|_{\cal H})={\cal O}(|\lambda|^{-2}\|f\|_{\cal H}).$$
This is integrable and changing the contour, we obtain
\begin{align*}
\tilde J_2(t)=\frac{1}{\sqrt{2\pi}}\int_{{\cal C}}\left(\frac{1}{i\lambda}\tilde\phi_B(A)R(i\lambda)A{\widehat{\psi'Uf}(\lambda)}_0+\tilde\phi_B(A)R(i\lambda){\widehat{\psi'Uf}(\lambda)}_1\right)d\lambda+{\cal O}(e^{-\varepsilon t}\|f\|_{\cal H})
\end{align*}
in ${ H}$. Thus for sufficient large $t$, we have
$$
\tilde \phi_B(A)u(t)=\frac{1}{\sqrt{2\pi}}\int_{{\cal C}}\tilde \phi_B(A)e^{i\lambda t}\left(R(i\lambda)(i\lambda+B){\widehat{\psi'Uf}(\lambda)}_0+R(i\lambda){\widehat{\psi'Uf}(\lambda)}_1\right)d\lambda+{\cal O}(e^{-\varepsilon t}\|f\|_{\cal H}).
$$
Changing $\psi$ in the above identity to the sequence $\{\psi_n\}$ and we take $n$ to infinity. Since $\|\psi_n'\|_{L^1}=\|\psi'\|_{L^1}$ and $\supp \ \psi_n\subset [0,c]$, the exponential term decays uniformly with respect to $n$ and ${\cal C}$ is compact, we can take limit under the integral sign. So tending $n$ to infinity, we conclude
\begin{equation}\label{asy}
\begin{split}
\tilde \phi_B(A)u(t)&=\frac{1}{2\pi}\int_{{\cal C}}\tilde \phi_B(A)e^{i\lambda t}\left(R(i\lambda)(i\lambda+B)u_0+R(i\lambda)u_1\right)d\lambda+{\cal O}(e^{-\varepsilon t}\|f\|_{\cal H})\\
&=\tilde J_3(t)+{\cal O}(e^{-\varepsilon t}\|f\|_{\cal H}).
\end{split}
\end{equation}
Finally we apply the identity $(\ref{asym resolvent})$ to get the asymptotic profile. By $(\ref{asym resolvent})$, we have
\begin{equation*}
\begin{split}
\tilde J_3(t)&=\frac{1}{2\pi}\int_{{\cal C}}B^{-1/2}\tilde \phi(\tilde A)(i\lambda+\tilde A)^{-1}B^{-1/2}e^{i\lambda t}\left(Bu_0+u_1\right)d\lambda\\
&+\frac{1}{2\pi}\int_{{\cal C}}B^{-1/2}\tilde \phi(\tilde A)(i\lambda+\tilde A)^{-1}B^{-1/2}e^{i\lambda t}\left(i\lambda u_0\right)d\lambda\\
&+\frac{1}{2\pi}\int_{{\cal C}}B^{-1/2}\tilde \phi(\tilde A)(i\lambda+\tilde A)^{-1} B^{-1/2}\lambda^2R(i\lambda)e^{i\lambda t}\left((i\lambda+B)u_0+u_1\right)d\lambda.
\end{split}
\end{equation*}
In the former integral, we can add contours $\tilde{\cal C}_\pm=\{\pm s+is; s\in[\varepsilon,\8)\}$, modulo exponential decaying factor. To the latter integrals, we apply same contour deformation as in Proposition $\ref{low decay}$. Then on the contour, $\tilde \phi(\tilde A)(i\lambda+\tilde A)^{-1}$ and $R(i\lambda)$ are $|\lambda|^{-1}$ order. So we obtain these integrals have $1/t$ bound. Thus we get
\begin{equation}\label{es1}
\begin{split}
\tilde J_3(t)&=\frac{1}{2\pi}\int_{{\cal C}\cup\tilde{\cal C}_+\cup\tilde{\cal C}_-}B^{-1/2}\tilde \phi(\tilde A)(i\lambda+\tilde A)^{-1}B^{-1/2}e^{i\lambda t}\left(Bu_0+u_1\right)d\lambda+{\cal O}(1/t)(\|u_0\|_{H}+\|u_1\|_H)\\
&=\tilde J_4(t)+{\cal O}(1/t)(\|u_0\|_{H}+\|u_1\|_H).
\end{split}
\end{equation}
Since $\tilde A$ is self-adjoint operator, it is a generator of an analytic semi-group and we have
$$\frac{1}{2\pi}\int_{{\cal C}_o\cup\tilde{\cal C}_+\cup\tilde{\cal C}_-}(i\lambda+\tilde A)^{-1}e^{i\lambda t}d\lambda=e^{-t\tilde A}.$$
Thus we obtain
\begin{equation}\label{es2}
\begin{split}
\tilde J_4(t)=&B^{-1/2}e^{-t\tilde A}B^{-1/2}\left(Bu_0+u_1\right)\\
&+\frac{1}{2\pi}\int_{{\cal C}_o\cup{\cal C}_+\cup{\cal C}_-}B^{-1/2}(1-\tilde \phi(\tilde A))(i\lambda+\tilde A)^{-1}B^{-1/2}e^{i\lambda t}\left(Bu_0+u_1\right)d\lambda.
\end{split}
\end{equation}
Since $(\ref{cutft})$, $1-\tilde \phi(\tilde A)=\phi(\tilde A)$ cut low frequency, $\phi(\tilde A)(i\lambda+\tilde A)^{-1}$ is analytic near the origin. Thanks to the $e^{it\lambda}$ factor, by the contour deformation across the origin, we know the latter decays exponentially. By $(\ref{asy})$, $(\ref{es1})$, $(\ref{es2})$ and density, we conclude
\begin{proposition}\label{low}
If $\phi$ cut high frequency, then there exists $C>0$ such that the following bound holds for $t>1$ and $u(t)$ which is the solution of the problem $(\ref{Cauchy})$ with the initial deta ${}^t(u_0,u_1)\in {\cal H}$
$$\|(1-\phi_B(A))u(t)- B^{-1/2}e^{-t\tilde A}B^{-1/2}\left(Bu_0+u_1\right)\|_{H}\leq C(t^{-1}(\|u_0\|_{H}+\|u_1\|_H)+ t^{-2}\|u_0\|_{H_1}).$$
\end{proposition}
From Proposition $\ref{high}$ and $\ref{low}$, we have proved Theorem ${\ref{main}}$. 
\begin{remark}\label{improve}
From the proof, we can get more sharp decay i.e. we can change $t^{-1}\|\sqrt A u_0\|_{H}$ term to $t^{-2}\|\sqrt A u_0\|_{H}$. Moreover if we apply the argument of Proposition $\ref{low exp}$, we can get this term is $t^{-\8}$ order.
\end{remark}
Next we show the optimality and an improvement of Theorem ${\ref{main}}$. The following statement says the optimality of the decay rate of Theorem ${\ref{main}}$.
\begin{proposition}\label{optimal}
Assume that $0$ belongs to the spectrum of $A$. Then
$$\limsup_{t\to\8}\sup_{\|{}^t(u_0,u_1)\|_{\cal H}\leq 1}t\| u(t)-B^{-1/2}e^{-t\tilde A}B^{1/2}u_0-B^{-1/2}e^{-t\tilde A}B^{-1/2}u_1\|_H>0$$
where $u(t)$ is the solution of the problem $(\ref{Cauchy})$.
\end{proposition}
Theorem ${\ref{main}}$ is an uniform estimate and for individual solution, we can obtain the following result.
\begin{proposition}\label{individual}
For each solution $u(t)$ of the problem $(\ref{Cauchy})$ with ${}^t(u_0,u_1)\in {\cal H}$, we have
$$\lim_{t\to\8}t\| u(t)-B^{-1/2}e^{-t\tilde A}B^{1/2}u_0-B^{-1/2}e^{-t\tilde A}B^{-1/2}u_1\|_H=0.$$
\end{proposition}
\begin{proof}[Proof of Proposition $\ref{individual}$] 
From the proof of Theorem ${\ref{main}}$, the main part of $u(t)$ is $\tilde J_3(t)$. Since $(i\lambda+\tilde A)^{-1} B^{-1/2}\lambda^2R(i\lambda)i\lambda={\cal O}(|\lambda|)$ on ${\cal C}$, we can apply the same argument as in the proof of Proposition $\ref{low decay}$ and $\ref{high}$, this part is negligible in $\tilde J_3$. So we have
\begin{align*}
I(t)&=u(t)-B^{-1/2}e^{-t\tilde A}B^{1/2}u_0-B^{-1/2}e^{-t\tilde A}B^{-1/2}u_1\\
&=\frac{1}{2\pi}\int_{{\cal C}}B^{-1/2}\tilde \phi(\tilde A)(i\lambda+\tilde A)^{-1}B^{-1/2}e^{i\lambda t}\left(i\lambda u_0\right)d\lambda\\
\ &+\frac{1}{2\pi}\int_{{\cal C}}B^{-1/2}\tilde \phi(\tilde A)(i\lambda+\tilde A)^{-1} B^{-1/2}\lambda^2R(i\lambda)e^{i\lambda t}\left(B u_0+u_1\right)d\lambda+{\cal O}(t^{-2}\|f\|_{\cal H}).
\end{align*}
For latter integral, we apply $(\ref{asym resolvent})$. Since $R(i\lambda)$ and $(Bi\lambda+A)^{-1}$ are ${\cal O}(|\lambda|^{-1})$ on ${\cal C}$, we can neglect the remainder term. For former integral, we apply the similar argument to get heat type asymptotic behavior. We obtain 
\begin{align*}
I(t)&=-B^{-1/2}\tilde Ae^{-{\tilde A}t }B^{-1/2} u_0+\frac{1}{2\pi}\int_{{\cal C}}B^{-1/2}\tilde \phi(\tilde A)(i\lambda+\tilde A)^{-1} B^{-1}\lambda^2(i\lambda+\tilde A)^{-1}e^{i\lambda t}B^{-1/2}\left(B u_0+u_1\right)d\lambda\\
&+{\cal O}(t^{-2}\|f\|_{\cal H}).
\end{align*}
For $\phi$ satisfying $(\ref{cutft})$, we define
$\phi_n(\tilde A)=\phi(n \tilde A)$. Then by the spectral theorem, we have 
$$\phi_n(\tilde A) u\to u \ \text{ in } H \ \text{ as } n\to \8.$$
We  decompose $I(t)$ as follows
\begin{align*}
I(t)&=-B^{-1/2}\tilde Ae^{-{\tilde A}t }B^{-1/2} u_0\\
&\quad +\frac{1}{2\pi}\int_{{\cal C}}B^{-1/2}\tilde \phi(\tilde A)(i\lambda+\tilde A)^{-1} B^{-1}\lambda^2(i\lambda+\tilde A)^{-1}e^{i\lambda t}\phi_n(\tilde A) B^{-1/2}\left(B u_0+u_1\right)d\lambda\\
&\quad +\frac{1}{2\pi}\int_{{\cal C}}B^{-1/2}\tilde \phi(\tilde A)(i\lambda+\tilde A)^{-1} B^{-1}\lambda^2(i\lambda+\tilde A)^{-1}e^{i\lambda t}(1-\phi_n(\tilde A))B^{-1/2}\left(B u_0+u_1\right)d\lambda\\
&\quad +{\cal O}(t^{-2}\|f\|_{\cal H}).
\end{align*}
Since $\phi_n(\tilde A)$ cut low frequency, $(i\lambda+\tilde A)\phi_n(\tilde A)$ is bounded near the origin. So in the first integral, the integrand is $|\lambda|$ order and can be estimate $C_nt^{-2}$ as in the proof of Proposition $\ref{low decay}$ and $\ref{high}$. 
The second integral can be estimated as in the estimate of $(\ref{es1})$ and we get uniform $t^{-1}$ estimate. Since $(1-\phi_n(\tilde A))B^{-1/2}\left(B u_0+u_1\right) \to 0$ in $H$ as $n\to H$. By letting $n$ sufficiently large, this part decays $\varepsilon t^{-1}$ order for any $\varepsilon>0$. Now we recall the fact,
$t\tilde Ae^{-\tilde At }u \to 0$ as $t\to\8$, for individual $u$. This fact is easily observed from the spectral theory. Thus we have proved.  
\end{proof}

\begin{proof}[Proof of Proposition $\ref{optimal}$] 
For simplicity, we can assume $u_0=0$. From the proof of Proposition $\ref{individual}$, we have
\begin{align*}
&u(t)-B^{-1/2}e^{-t\tilde A}B^{-1/2}u_1\\
&=\frac{1}{2\pi}\int_{{\cal C}}B^{-1/2}\tilde \phi(\tilde A)(i\lambda+\tilde A)^{-1} B^{-1/2}\lambda^2 B^{-1/2}(i{\lambda}+\tilde A)^{-1}B^{-1/2}e^{i\lambda t}u_1d\lambda+{\cal O}(t^{-2}\|u_1\|_{ H}).
\end{align*}
So we get
\begin{align*}
&(u(t)-B^{-1/2}e^{-t\tilde A}B^{-1/2}u_1,u_1)_H\\
&=\frac{1}{2\pi}\int_{{\cal C}}(B^{-1}\lambda^2e^{i\lambda t}(i{\lambda}+\tilde A)^{-1}B^{-1/2}u_1,\tilde \phi(\tilde A)(i\overline{\lambda}+\tilde A)^{-1} B^{-1/2}u_1 )_Hd\lambda+{\cal O}(t^{-2}\|u_1\|^2_{ H}).
\end{align*}
For $a\in \sigma (\tilde A)\cap \supp \ \tilde \phi$, we take a sequence of functions $f_n$ such that $\psi_n(\tilde A) f_n=f_n$ and $\|f_n\|_H=1$. Here 
$$\psi_n(x)=\begin{cases}
1, \ \ \ \ \  |x-a|\leq 1/n,\\
0, \ \ \ \ \ |x-a|>1/n.
\end{cases}
$$
We take $u_1^n=B^{1/2}f_n$, then
$$\frac{1}{C}=\frac{1}{C}\|B^{-1/2}u_1^n\|_H\leq \|B^{-1}u_1^n\|_{H}\leq C\|B^{-1/2}u_1^n\|_H=C.$$
So if necessary, we take a subsequence and we can assume $\|B^{-1}u_1^n\|_{H}$ converge to $\beta_a\geq 1/C$ as $n$ tends to infinity. We also note $\|A f_n\|_H$ is uniformly bounded so from the continuity of $B^{1/2}$ and interpolation, $\|u_1^n\|_{H_1}$ is  uniformly bounded.  
For $\lambda\in {\cal C}$, we easily see
$$(i{\lambda}+\tilde A)^{-1}B^{-1/2}u_1^n-(i{\lambda}+a)^{-1}B^{-1/2}u_1^n\to 0$$
in $H$ as $n$ tends to infinity and the convergence is locally uniform. We have
\begin{align*}
\lim_{n\to\8} (u^n(t)-B^{-1/2}e^{-t\tilde A}B^{-1/2}u_1^n,u_1^n)_H=\beta_a\frac{1}{2\pi}\int_{{\cal C}}\lambda^2e^{i\lambda t}(i{\lambda}+a)^{-2}  d\lambda+{\cal O}(t^{-2}).
\end{align*}
In the above integral, we can change $\cal C$ to a simple closed contour $\tilde{\cal C}$ by adding a contour with exponential error terms. 
So we have 
\begin{align*}
\lim_{n\to\8} (u^n(t)-B^{-1/2}e^{-t\tilde A}B^{-1/2}u_1^n,u_1^n)_H&=i\beta_a\frac{1}{2\pi} \int_{\tilde{\cal C}}\lambda^2e^{i\lambda t}\pl_\lambda(i{\lambda}+a)^{-1} d\lambda+{\cal O}(t^{-2})\\
&=-\beta_a \frac{1}{2\pi i }\int_{\tilde{\cal C}}( 2i \lambda e^{i\lambda t}-\lambda^2t e^{i\lambda t} ) (i{\lambda}+a)^{-1}d(i\lambda)+{\cal O}(t^{-2})\\
&=\beta_a(2a e^{-at}-a^2t e^{-a t})+{\cal O}(t^{-2}).
\end{align*}
Taking $t= 1/a$ and $a\to 0$. Since $\beta_a \geq 1/C$, we obtain this integral does not decay faster than $t^{-1}$ order. So we have proved Proposition $\ref{optimal}$.
\end{proof}

\end{section}

\begin{section}{Proof of Theorem {\ref{main 2}}}
 Here we study the asymptotic expansion of the propagator for proving Theorem ${\ref{main 2}}$. In this section, we assume $H=L^2$.
We apply the identity $(\ref{asym resolvent})$ to $(\ref{asy})$ and get 
\begin{equation}
\begin{split}
\tilde \phi_B(A)u(t)=&\frac{1}{2\pi}\int_{{\cal C}}B^{-1/2}\tilde \phi(\tilde A)(i\lambda+\tilde A)^{-1}B^{-1/2}e^{i\lambda t}\left(Bu_0+u_1\right)d\lambda
\\
&\quad +\frac{1}{2\pi}\int_{{\cal C}}B^{-1/2}\tilde \phi(\tilde A)(i\lambda+\tilde A)^{-1}B^{-1/2}e^{i\lambda t}\left(i\lambda u_0\right)d\lambda\\
&\quad +\frac{1}{2\pi}\int_{{\cal C}}B^{-1/2}\tilde \phi(\tilde A)(i\lambda+\tilde A)^{-1}B^{-1/2}R(i\lambda)e^{i\lambda t}\left(Bu_0+u_1\right)d\lambda
\\
&\quad +\frac{1}{2\pi}\int_{{\cal C}}B^{-1/2}\tilde \phi(\tilde A)(i\lambda+\tilde A)^{-1}B^{-1/2}R(i\lambda)e^{i\lambda t}\left(i\lambda u_0\right)d\lambda+{\cal O}(e^{-\varepsilon t}\|f\|_{\cal H})\\
&=I+{\tilde I}+ R_1+\tilde R_1+{\cal O}(e^{-\varepsilon t}\|f\|_{\cal H})=I+{\tilde I}+R+{\cal O}(e^{-\varepsilon t}\|f\|_{\cal H}).
\end{split}
\end{equation}
$R=R_1+\tilde R_1$ is the reminder term and ${\cal C}$ is a contour around the origin as in the proof of Proposition $\ref{low decay}$. We have estimated $I$ in previous section and get heat type asymptotic behavior modulo exponential decay. 
\begin{equation}
I=B^{-1/2}e^{-t\tilde A}B^{-1/2} (Bu_0+u_1)+{\cal O}(e^{-\varepsilon t}\|f\|_{\cal H}).
\end{equation}
So we shall estimate $\tilde I$ and $R$. Since $(\lambda+\tilde A)^{-1}(\lambda+\tilde A)=Id$, we have
$$\tilde I=\frac{1}{2\pi} \int_{{\cal C}}B^{-1/2}\tilde \phi(\tilde A)B^{-1/2}e^{i\lambda t} u_0d\lambda-\frac{1}{2\pi}\int_{{\cal C}}B^{-1/2}\tilde \phi(\tilde A)\tilde A (i\lambda+\tilde A)^{-1}B^{-1/2}e^{i\lambda t} u_0d\lambda$$
In the former integral, the integrand is holomorphic. So by the contour deformation across the origin, we get exponential decay from $e^{it\lambda}$ term. In the latter integral, we can apply similar argument for proving heat type asymptotic behavior and obtain the following asymptotics
\begin{equation}\label{tilde I}
\tilde I=-B^{-1/2}\tilde A e^{-t\tilde A}B^{-1/2} u_0 +{\cal O}(e^{-\varepsilon t}\|f\|_{\cal H}).
\end{equation}
By the assumption, we have the bound
$$e^{-t/2\tilde A}: L^1 \to L^2={\cal O}(t^{-m}).$$
Since $xe^{-x},\ x\geq 0$ is bounded, $tAe^{-tA}$ is bounded. So we have
$$\tilde A e^{-t/2\tilde A}:L^2 \to L^2= {\cal O}(t^{-1}).$$
$B^{-1/2}$ is also a bounded operator on $L^1$ so if $u_0\in L^1$, using the identity 
$\tilde A e^{-t\tilde A}= \tilde A e^{-t/2\tilde A} e^{-t/2\tilde A}$, we obtain
\begin{equation}
\|\tilde I\|_{L^2}={\cal O}(t^{-m-1} (\|u_0\|_{L^1}+\|f\|_{\cal H})).
\end{equation}
Next we study the remainder term $R$. Since $\lambda^2B^{-1/2}(i\lambda+\tilde A)^{-1}B^{-1/2}: H\to H ={\cal O}(|\lambda|)$ on the contour, using  $(\ref{asym resolvent})$ repeatedly, we obtain, for sufficiently large $N$ 
\begin{align*}
R=&\frac{1}{2\pi}\int_{{\cal C}}e^{i\lambda t}B^{-1/2}\tilde \phi(\tilde A)(i\lambda+\tilde A)^{-1}B^{-1}\lambda^2(i\lambda+\tilde A)^{-1}B^{-1/2}\left(Bu_0+u_1\right)d\lambda\\
&+\frac{1}{2\pi}\int_{{\cal C}}e^{i\lambda t}B^{-1/2}\tilde \phi(\tilde A)(i\lambda+\tilde A)^{-1}B^{-1}\lambda^2(i\lambda+\tilde A)^{-1}B^{-1/2}\ i\lambda u_0\ d\lambda\\
&+\frac{1}{2\pi}\int_{{\cal C}}e^{i\lambda t}B^{-1/2}\tilde \phi(\tilde A)(i\lambda+\tilde A)^{-1}B^{-1}\lambda^2(i\lambda+\tilde A)^{-1}B^{-1}\lambda^2(i\lambda+\tilde A)^{-1}B^{-1/2}\left(Bu_0+u_1\right)d\lambda\\
&+\cdots+\int_{{\cal C}}e^{i\lambda t}{\cal O}(|\lambda|^N)\|f\|_{\cal H}d\lambda
=\frac{1}{2\pi}K_1+\frac{1}{2\pi}\tilde K_1+\frac{1}{2\pi}K_2+\cdots+\frac{1}{2\pi}K_N+\frac{1}{2\pi}\tilde K_N+{\cal O}(t^{-N-1})\|f\|_{\cal H}.
\end{align*}
For the latter estimate, we apply same contour deformation as in Proposition $\ref{low decay}$ by using the estimate ${\cal O}(|\lambda|^N)$ on the contour. Here
\begin{align*}
K_n=&\int_{{\cal C}}e^{i\lambda t}B^{-1/2}\tilde \phi(\tilde A)(i\lambda+\tilde A)^{-1}B^{-1/2}(\lambda^2B^{-1/2}(i\lambda+\tilde A)^{-1}B^{-1/2}))^n\left(Bu_0+u_1\right)d\lambda,\\
{\tilde K}_n=&\int_{{\cal C}}e^{i\lambda t}B^{-1/2}\tilde \phi(\tilde A)(i\lambda+\tilde A)^{-1}B^{-1/2}(\lambda^2B^{-1/2}(i\lambda+\tilde A)^{-1}B^{-1/2}))^n\ i\lambda u_0\ d\lambda.
\end{align*}
First we estimate $K_1$ and $\tilde K_1$.
\begin{proposition}\label{K1}
There exists a constant $C>0$ such that the following estimates hold for any $t>1$ and $u_0\in L^1 \cap D(\sqrt A)$, $u_1\in L^2\cap L^1$ 
\begin{align*}
&\|K_1\|_{L^2}\leq C(t^{-m-1} (\|u_0\|_{L^1\cap L^2}+\|u_1)\|_{L^1\cap L^2}),\\
&\|\tilde K_1\|_{L^2}\leq C(t^{-m-2} (\|u_0\|_{L^1\cap L^2}+\|u_1\|_{L^1\cap L^2}).
\end{align*}
\end{proposition}
\begin{proof}
We use the following identity
\begin{equation}\label{re}
(i\lambda+\tilde A)^{-1}=(i\lambda+\tilde A)^{-1}e^{-t(i\lambda+\tilde A)/2}+\int^{t/2}_0e^{-t_1(i\lambda+\tilde A)}dt_1.
\end{equation}
Inserting this identity, we obtain 
\begin{align*}
K_1&=\int_{\cal C}e^{i\lambda t/2}B^{-1/2}\tilde \phi(\tilde A)(i\lambda+\tilde A)^{-1}B^{-1}\lambda^2(i\lambda+\tilde A)^{-1} e^{-t\tilde A/2}B^{-1/2}\left(Bu_0+u_1\right)d\lambda\\
&\quad+\int_{{\cal C}}d\lambda \int^{t/2}_0dt_1\ e^{i\lambda (t-t_1)}B^{-1/2}\tilde \phi(\tilde A)\lambda^2(i\lambda+\tilde A)^{-1}B^{-1}e^{-t_1\tilde A}B^{-1/2}\left(Bu_0+u_1\right).
\end{align*}
The first integral is easily treated since we have the $(i\lambda+\tilde A)^{-1}={\cal O}(|\lambda|^{-1})$ on $\cal C$. 
For latter integrals, we apply the following identity 
$$\lambda(i\lambda+\tilde A)^{-1}=-iId+i\tilde A(i\lambda+\tilde A)^{-1}.$$
So we have
\begin{align*}
K_1&=\int_{{\cal C}}e^{i\lambda t/2}{\cal O}(1)e^{-t/2\tilde A}B^{-1/2}\left(Bu_0+u_1\right)d\lambda\\
&\quad -i\int_{{\cal C}}d\lambda \int^{t/2}_0dt_1\ e^{i\lambda (t-t_1)}B^{-1/2}\tilde \phi(\tilde A)\lambda B^{-1}e^{-t_1\tilde A}B^{-1/2}\left(Bu_0+u_1\right)\\
&\quad +\int_{{\cal C}}d\lambda\int^{t/2}_0dt_1\ e^{i\lambda (t-t_1)}B^{-1/2}\tilde \phi(\tilde A)\tilde AB^{-1}e^{-t_1\tilde A}B^{-1/2}\left(Bu_0+u_1\right)\\
&\quad -\int_{{\cal C}}d\lambda\int^{t/2}_0dt_1\ e^{i\lambda (t-t_1)}B^{-1/2}\tilde \phi(\tilde A)\tilde A^2(i\lambda+\tilde A)^{-1}B^{-1}e^{-t_1\tilde A}B^{-1/2}\left(Bu_0+u_1\right).
\end{align*}
In second and third integrals, integrands are holomorphic with respect to $\lambda$ variable, and $t-t_1\geqq t/2$, we can deform the contour across the origin and get exponential decay. For the fourth integral, we can get heat type semi-group modulo exponential error by applying same argument as in the previous section. We obtain
\begin{align*}
K_1&=\int_{{\cal C}}e^{i\lambda t/2}{\cal O}(1)d\lambda e^{-t/2\tilde A}B^{-1/2}\left(Bu_0+u_1\right)\\
&\quad+\int^{t/2}_0 B^{-1/2}\phi(\tilde A)\tilde A^2 e^{-(t-t_1)\tilde A}B^{-1}e^{-t_1\tilde A}B^{-1/2}\left(Bu_0+u_1\right)dt_1+{\cal O}(e^{-\varepsilon t})(\|u_0\|_{L^2}+\|u_1\|_{L^2})\\
&={\tilde I}_1+{\tilde I}_2+{\cal O}(e^{-\varepsilon t})(\|u_0\|_{L^2}+\|u_1\|_{L^2}).
\end{align*}
Applying same contour deformation as in Proposition $\ref{low decay}$, we get
\begin{align*}
\|{\tilde I}_1\|_{L^2}&={\cal O}(t^{-1}\| e^{-t/2\tilde A}B^{-1/2}\left(Bu_0+u_1\right)\|_{L^2})\\
&={{\cal O}}(t^{-m-1})(\|u_0\|_{L^1}+\|u_1\|_{L^1}).
\end{align*}
Since $\int^{t/2}_0dt_1={\cal O}(t)$ and $\|\phi(\tilde A) \tilde A^2 e^{-t/4\tilde A}\|_{L^2\to L^2}={\cal O}(t^{-2})$, the following estimate holds 
\begin{align*}
\|{\tilde I}_2\|_{L^2}&={\cal O}(t\|\phi(\tilde A) \tilde A^2 e^{-t/4\tilde A}\|_{L^2\to L^2}\sup_{0\leq t_1\leq t/2} \| e^{-\tilde A(3t/4-t_1)}B^{-1}e^{-t_1\tilde A}B^{-1/2}\left(Bu_0+u_1\right)\|_{L^2})\\
&={\cal O}(t^{-1}\sup_{0\leq t_1\leq t/2} \| e^{-\tilde A(3t/4-t_1)}\|_{L^1\to L^2}\|B^{-1}e^{-t_1\tilde A}B^{-1/2}\left(Bu_0+u_1\right)\|_{L^1})\\
&={\cal O}(t^{-m-1} (\|u_0\|_{L^1}+\|u_1)\|_{L^1}).
\end{align*}
Thus we have estimated $K_1$ part. The estimate of $\tilde K_1$ is almost similar.
\end{proof}
We can apply similar idea for the estimate of $K_n$ and $\tilde K_n$ terms, as an example, we next prove the following estimate
$$\|K_2\|_{L^2}={\cal O}(t^{-m-2} (\|u_0\|_{L^1\cap L^2}+\|u_1\|_{L^1\cap L^2})).$$
By the identity $(\ref{re})$, we have
\begin{align*}
K_2&=\int_{{\cal C}}e^{i\lambda t/2}B^{-1/2}\tilde \phi(\tilde A)(i\lambda+\tilde A)^{-1}B^{-1}\lambda^2(i\lambda+\tilde A)^{-1}B^{-1}\lambda^2(i\lambda+\tilde A)^{-1}e^{-t/2\tilde A}B^{-1/2}\left(Bu_0+u_1\right)d\lambda\\
&\quad+\int_{{\cal C}}d\lambda\int^{t/2}_0dt_1\ e^{i\lambda (t-t_1)}B^{-1/2}\tilde \phi(\tilde A)(i\lambda+\tilde A)^{-1}B^{-1}\lambda^2(i\lambda+\tilde A)^{-1}B^{-1}\lambda^2e^{-t_1\tilde A}B^{-1/2}\left(Bu_0+u_1\right).
\end{align*}
Applying the following identity to the second integral
\begin{equation}\label{re2}
  (i\lambda+\tilde A)^{-1}=(i\lambda+\tilde A)^{-1}e^{-t/4(i\lambda+\tilde A)}+\int^{t/4}_0e^{-t_2(i\lambda+\tilde A)}dt_2,
\end{equation}
we get
\begin{align*}
K_2&=\int_{{\cal C}}e^{i\lambda t/2}{\cal O}(|\lambda|)d\lambda e^{-t/2\tilde A}B^{-1/2}\left(Bu_0+u_1\right)\\
&\quad+\int_{{\cal C}}d\lambda\int^{t/2}_0dt_1\ e^{i\lambda (3t/4-t_1)}B^{-1/2}\tilde \phi(\tilde A)(i\lambda+\tilde A)^{-1}B^{-1}\lambda^4 (i\lambda+\tilde A)^{-1} e^{-t/4\tilde A}B^{-1}e^{-t_1\tilde A}B^{-1/2}\left(Bu_0+u_1\right)\\
&\quad+\int_{{\cal C}}d\lambda \int^{t/2}_0dt_1 \int^{t/4}_0dt_2 \ e^{i\lambda (t-t_1-t_2)}B^{-1/2}\tilde \phi(\tilde A)\lambda^4(i\lambda+\tilde A)^{-1}B^{-1} e^{-t_2\tilde A}B^{-1}e^{-t_1\tilde A}B^{-1/2}\left(Bu_0+u_1\right).
\end{align*}
We can apply similar argument as in the proof of Proposition $\ref{K1}$ and we can get the estimate. 

In the same way, we obtain
\begin{proposition}
There exists a constant $C>0$ such that the following estimates holds for any $t>1$ and $u_0\in L^1 \cap D(A)$, $u_1\in L^2\cap L^1$ 
\begin{align*}
&\|K_n\|_{L^2}\leq C(t^{-m-n} (\|u_0\|_{L^1\cap L^2}+\|u_1\|_{L^1\cap L^2}),\\
&\|\tilde K_n\|_{L^2}\leq C(t^{-m-n-1} (\|u_0\|_{L^1\cap L^2}+\|u_1\|_{L^1\cap L^2}).
\end{align*}
\end{proposition}
By these propositions, we have estimated $\tilde \phi_B(A)u(t)$ part and we have
\begin{equation}
\tilde \phi_B(A)u(t)=B^{-1/2}e^{-t\tilde A}B^{-1/2} (Bu_0+u_1)+{\cal O}(t^{-m-1}\|u_0\|_{L^1}+\|u_1\|_{L^1}+\|f\|_{\cal H}).
\end{equation}
Next we estimate  $(1-\tilde \phi_B(A))u(t)=\phi_B(A)u(t)$ term by sharpening the estimate of Proposition $\ref{high}$. In the proof of Proposition $\ref{high}$, we have the decomposition
$$\phi_B(A)\psi(t)U(t)f=\frac{1}{2\pi}J_2f+{\cal O}(e^{-\varepsilon t}\|f\|_{\cal H}).$$
Here $\psi (t)=1$ for $t>2$ and 
\begin{equation*}
J_2f=\int_{\cal C}d\lambda \int^\8_{-\8} ds\  B^{-1/2}\lambda^2\phi(\tilde A)(i\lambda+\tilde A)^{-1}e^{i\lambda (t-s)} B^{-1/2}\psi'(s)U(s)(i\lambda-{\cal A})^{-1}{f}.
\end{equation*}
Since $\cal A$ is a generator of $U(t)$, we exchange $U(t)$ to $(i\lambda- {\cal A})^{-1}$. We apply similar argument to Proposition $\ref{K1}$. By $(\ref{asym resolvent})$ and $(\ref{reso})$, we have
\begin{equation}
\begin{split}
J_2f=\int_{\cal C}d\lambda \int^\8_{-\8}\psi'(s)ds\  B^{-1/2}\lambda^2\phi(\tilde A)(i\lambda+\tilde A)^{-1}e^{i\lambda (t-s)} B^{-1/2}U(s)B^{-1/2}(i\lambda +\tilde A)^{-1}B^{-1/2}M(\lambda) {f}+R_2.
\end{split}
\end{equation}
Here $R_2$ is the reminder term. We recall the support of $\psi'$ is in $[1,2]$. We apply identity $(\ref{re})$ and $J_2f$ becomes
\begin{equation}
\begin{split}
&\int_{\cal C}d\lambda \int^2_{1}\psi'(s)ds\  B^{-1/2}\lambda^2\phi(\tilde A)(i\lambda+\tilde A)^{-1}e^{i\lambda (t/2-s)} B^{-1/2}U(s)B^{-1/2}(i\lambda+\tilde A)^{-1}e^{-t\tilde A/2}B^{-1/2}M(\lambda) {f}\\
&+\int_{\cal C}d\lambda  \int^2_{1}\psi'(s)ds\int^{t/2}_0dt_1\  B^{-1/2}\lambda^2\phi(\tilde A)(i\lambda+\tilde A)^{-1}e^{i\lambda (t-t_1-s)} B^{-1/2}U(s)B^{-1/2}e^{-t_1\tilde A/2}B^{-1/2}M(\lambda) {f}+R_2.
\end{split}
\end{equation}
For the first integral, on the contour, we have
$$\|(i\lambda +\tilde A)^{-1}\|_{L^2 \to H^1}={\cal O}(|\lambda|^{-1}),$$
and 
$$\|e^{-t{\tilde A}}B^{-1/2}A\|_{L^1\to L^2}=\|\tilde A e^{-t{\tilde A}}B^{1/2}\|_{L^1\to L^2}={\cal O}(t^{-m-1}). $$
From these estimate, we can estimate the first integral by similar contour deformation as in Proposition $\ref{low decay}$ and get $t^{-m-2}$ order decay. 
In the second integral, we need the following estimate
$$\|B^{-1/2}e^{-t_1 \tilde A/2}B^{-1/2}M(\lambda) {f}\|_{\cal H}={\cal O}(t_1^{-1/2}+1)\|f\|_{\cal H}.$$
This estimate follows from
$$\|B^{-1/2}e^{-t_1\tilde A/2} {u}\|_{H_1}={\cal O}(t_1^{-1/2}+1)\|u\|_{H}$$
and 
$$\|B^{-1/2}e^{-t_1\tilde A/2}B^{-1/2}A {u}\|_{H}={\cal O}(t_1^{-1/2}+1)(\|u\|_{H_1}).$$
In the above, the first estimate holds by interpolating the following estimate and boundedness on $H$
$$\|A B^{-1/2}e^{-t_1\tilde A/2}B^{-1/2}u\|_H=\|B^{1/2} \tilde A e^{-t_1 \tilde A/2}B^{-1/2}u\|_H={\cal O}(t_1^{-1})\|u\|_{H}$$
The second one is obtained from the first estimate by taking its adjoint
\begin{align*}
\|B^{-1/2}e^{-t_1 \tilde A/2}B^{-1/2}Au\|_{H}&=\|B^{-1/2}e^{-t_1 \tilde A/2}B^{-1/2}A^{1/2}A^{1/2} u\|_{H}\\
&\leq \|(B^{-1/2}e^{-t_1 \tilde A/2}B^{-1/2}A^{1/2})\|_{H\to H} \|A^{1/2} {u}\|_{H}\\
&=\|(B^{-1/2}e^{-t_1 \tilde A/2}B^{-1/2}A^{1/2})^*\|_{H\to H} \|A^{1/2} {u}\|_{H}\\
&=\|A^{1/2}B^{-1/2}e^{-t_1/2 \tilde A}B^{-1/2} \|_{H\to H}\|{{\sqrt A}u}\|_{H}\\
&={\cal O}(t_1^{-1/2}+1)\|u\|_{H_1}.
\end{align*}
In the second integral, since the integrand is holomorphic and $t_1^{-1/2}$ is integrable near $t_1=0$, we obtain exponential decay by contour deformation. Using the above estimates, we can repeat similar argument for reminder term $R_2$. In this case, one may think the $L^1L^1$ estimate of  $e^{-t_1\tilde A/2}\tilde A$ is needed but since $e^{-t\tilde A/2}\tilde A=-\pl_t e^{-t\tilde A/2}$ and using  integration by parts, we can avoid this term. Since the argument is essentially similar, we omit the detail and we get
\begin{proposition}\label{low exp}
There exists a constant $C>0$ such that the following estimates holds for any $t>1$ and $f\in {\cal H}$ 
$$\|(1-\tilde \phi_B(A))u(t)\|_H\leq C t^{-m-2}\|f\|_{\cal H}.$$
\end{proposition}
From these propositions, we have proved Theorem {\ref{main 2}}.
\end{section} 

\begin{section}{Applications}
As applications of our approach, here we consider some examples of damped wave equation on $\R^d$. We study the following divergence form damped wave equation
\begin{equation}\label{euc damp}
\left(\pl_{t}^2-\sum_{i,j=1}^d\pl_ig_{ij}(x)\pl_j+a(x)\pl_t\right)u(x,t)=(\pl_{t}^2+P+a(x)\pl_t)u(x,t)=0.
\end{equation}
Here we write $\pl_i=\frac{\pl}{\pl x_i}$ and $P= -\sum_{i,j=1}^d\pl_ig_{ij}(x)\pl_j$. 
We assume $g_{ij}$ are smooth functions with bounded derivatives and $(g_{ij}(x))$ is a family of uniformly elliptic real symmetric matrices, i.e. there exists a constant $C>0$ such that 
$$\frac{1}{C} Id \leq (g_{ij}(x)) \leq C Id$$
holds for any $x\in \R^d$. 
Under this assumption $P$ is a positive definite self-adjoint operator on $H^2$ and $D(\sqrt P)=H^1$ where $H^1$ and $H^2$ are usual sobolev spaces on $\R^d$. 
The damping term $a(x)$ is a smooth non-negative function with bounded derivatives. If $a(x)>c$ for a constant $c>0$,  by Lemma $\ref{heat est} $ below, we can apply Theorem ${\ref{main}}$ and Theorem ${\ref{main 2}}$. Thus we obtain similar decay estimates for the constant coefficient case. 
\begin{theorem}
Let $u$ be a solution of $(\ref{euc damp})$ with the initial data $u|_{t=0}=u_0$ and $\pl_tu|_{t=0}=u_1$. We assume $a(x)>c$ for $c>0$ and define
$$v(t)=a(\cdot)^{-1/2}e^{-ta^{-1/2}Pa^{-1/2}}a(\cdot)^{1/2}u_0-a(\cdot)^{-1/2}e^{-ta^{-1/2}Pa^{-1/2}}a(\cdot)^{-1/2}u_1.$$
Then there exists $C>0$ such that we have the following asymptotic profile
\begin{align*}
\|u(t)-v(t)\|_{L^2}\leq C{t^{-1}}(\|u_0\|_{H^1}+\|u_1\|_{L^2} ).
\end{align*}
for $u_0\in H^1$, $u_1\in L^2$ and $t>1$. Furthermore, if we assume $u_0\in H^1\cap L^1$ and $u_1\in L^2\cap L^1$, then we also have the following bound for $t>1$, 
\begin{align*}
\|u(t)-v(t)\|_{L^2}\leq C{t^{-d/4-1}}(\|u_0\|_{L^2 \cap L^1}+\|u_1\|_{L^2 \cap L^1} ).
\end{align*}

\end{theorem}
We shall study the case that  $a(x)$ may be vanish on a compact set however the geometric control condition holds. So we recall about the geometric control condition. For the operator $P$, we define its Hamiltonian flow $\phi_t(x,\xi)=(x(t),\xi(t))$ by the solution of the following the Hamiltonian equation
\begin{equation*}
\begin{cases}
&\displaystyle \frac{d x}{dt}=\frac{\pl p}{\pl \xi},\ \ \displaystyle \frac{d \xi}{dt}=-\frac{\pl p}{\pl x}\\
&(x(0),\xi(0))=(x,\xi).
\end{cases}
\end{equation*} 
Here the symbol $p(x,\xi)=\sum_{i,j=1}^dg_{ij}(x)\xi_i\xi_j$ is associated with $P$.  
By the uniform ellipticity, the above Hamilton equation has time global solutions for any initial date and we can define the Hamiltonian flow.

For the damped wave equation, it is known that the mean value of the damping term along a geodesic is closely related to the asymptotic profile of the solution c.f. \cite{Jo}, \cite{Ra-Ta}. So we introduce
$$\langle a\rangle_T(x,\xi)=\frac{1}{T}\int^T_0a(\phi_t(x,\xi))dt$$
where $T>0$ and  we use the notion $a(x,\xi):=a(x)$. 
We study the asymptotic behavior of the damped wave equation under the following geometric control condition. 
\medskip

\noindent 
{\bf Geometric control condition}: There exist $T>0$ and $\alpha>0$ such that for all $(x,\xi)\in p^{-1}(1)$, we have $\langle a \rangle_T(x,\xi)\geq \alpha$.
\medskip

Under the geometric control condition, we can expect the asymptotic profile is similar to the strictly positive case. We first obtain the energy decay. Let $U(t)$ be the propagator of the damped wave equation. We have the following energy decay estimate.
\begin{theorem}\label{geo ene decay}
We assume the geometric control condition and $a(x)>c$ for $|x|>R$ where $c>0$ and $R>0$ are some constants. Let $u$ be a solution of $(\ref{euc damp})$ with the initial data $u|_{t=0}=u_0\in H^1$ and $\pl_tu|_{t=0}=u_1\in L^2$. 
Then there exists a constant $C$ such that for $t>1$ the following bound holds
$$\|\nabla u\|^2_{L^2}+\|\pl_t u\|^2_{L^2}\leq C(\sum_{i,j=1}^d(g_{ij}\pl_iu,\pl_ju)_{L^2}+\|\pl_t u\|^2_{L^2}) \leq C t^{-1}(\|u_0\|^2_{H^1}+\|u_1)\|_{L^2}).$$
\end{theorem}
In this case, the damping term is not uniformly positive. But we can prove the similar resolvent estimate and we can apply the argument of proving Theorem 4.1.  We study the resolvent by dividing three range of frequencies.

For low frequency, we use the following estimate
\begin{lemma}
If $\lambda$ and $\delta>0$ are sufficiently small, then $(P+\lambda^2+\lambda a(\cdot))^{-1}:L^2\to L^2$ exists for ${\rm Re} \lambda>0$ or $|{\rm Im}\lambda|>\delta|{\rm Re}\lambda|$ and we have the following bound
$$(P+\lambda^2+\lambda a(\cdot))^{-1}={\cal O}(|\lambda|^{-1}).$$
\end{lemma}
\begin{proof}
First we prove the following bound, for sufficient small $\lambda>0$
\begin{equation}\label{low bound eq}
\int_{\R^d} |\nabla u|^2dx+\lambda \int_{\R^d} a(x)|u(x)|^2 dx\geq C\lambda\int_{\R^d} |u|^2 dx.
\end{equation}
Taking a smooth compact support non-negative function $\chi(x)$, $\chi (x)=1$ for $|x|<1$. 
We set $\chi_r(x)=\chi(r x)$ for $r>0$. We have 
$$ \|\nabla u\|_{L^2}\geq \|\chi_r \nabla u\|_{L^2}\geq\| \nabla (\chi_r u)\|_{L^2}-\| [\nabla, \chi_r] u\|_{L^2}.$$
Since $a>c$ for $|x|>R$ and $(\nabla\chi)_r(x)= \nabla\chi(r x)=0$ for $|x|<1/r $, we obtain 
$$\| [\nabla, \chi_r] u\|_{L^2}=r \| (\nabla\chi)_r u\|_{L^2}\leq C r \left(\int_{\R^d} a(x)|u(x)|^2 dx\right)^{1/2}$$
for sufficient small $r$ and some $C>0$. Since $\chi_r u=0$ for $|x|>1/r$, by the Poincar\'e inequality, we have
$$\|\chi_r u\|_{L^2}\leq C\frac{1}{r}\| \nabla (\chi_r u)\|_{L^2}.$$
So we get
$$ \|\nabla u\|_{L^2}\geq \frac{1}{C} r \|\chi_r u\|_{L^2}- C r\left( \int_{\R^d} a(x)|u(x)|^2 dx\right)^{1/2}.$$
Taking $r=\varepsilon \sqrt{\lambda}$ for sufficient small $\varepsilon$, we have $(\ref{low bound eq})$ since $a>c$ for $|x|>R$. 
Now we prove the lemma. We use the following identity
$$((P+\lambda^2+\lambda a(\cdot) )u,u)_{L^2}=(P u,u)_{L^2}+\lambda^2\|u\|^2_{L^2}+\lambda \int_{\R^d} a(x)|u(x)|^2 dx,$$
From this identity and uniform ellipticity of $P$, we have
$$|((P+\lambda^2+\lambda a(\cdot) )u,u)_{L^2}|\geq C\|\nabla u\|_{L^2}+C| \lambda| \int_{\R^d} a(x)|u(x)|^2 dx-{\cal O}(|\lambda|^2\|u\|_{L^2}).$$
for ${\rm Re}\lambda>0$ or $|{\rm Im}\lambda|>\delta |{\rm Re}\lambda|$ if $\delta$ is sufficient small . 
So for sufficient small $|\lambda|$, applying $(\ref{low bound eq})$ we get
$$|((P+\lambda^2+\lambda a(\cdot) )u,u)_{L^2}|\geq C|\lambda| \|u\|^2_{L^2}.$$
From this bound we easily prove the existence of the inverse and the estimate by applying similar argument for Lemma \ref{lem3.1}. 
\end{proof}

\begin{remark}
The above estimate holds for $u\in C^\8_0(\Omega)$ for exterior domain $\Omega$ and we can generalize the theorem for exterior problem with Dirichlet boundary condition.
\end{remark}

By the geometric control condition, we can control high frequency part and we can prove the following statement. 
\begin{lemma} \label{highreso}Under the geometric control condition, 
$R(\lambda)=(P+\lambda^2+\lambda a(\cdot))^{-1}:L^2 \to L^2 $ exists for $\lambda=i\tau$, $\tau\in\R\backslash\{0\}$, $|\tau|>C$ for sufficient large $C>0$ and we have the following bound
$$(P+\lambda^2+\lambda a(\cdot))^{-1}={\cal O}(|\lambda|^{-1})$$
\end{lemma}
This type lemma is well-known, e.g. \cite{Jo}, \cite{Zw}, so we omit the proof. From this Lemma and by Proposition $\ref{resolvent(prop)}$, we have
$$(\lambda-{\cal A})^{-1}={\cal O}(1+|\lambda|^{-1}).$$
This estimate guarantees the analytic continuation of the resolvent across the imaginary axis for high frequency region. 

For middle frequency region, we can use the Fredholm theory. We follow the argument in $\cite{Zw}$.
We take a smooth compact support function $b(x)$ satisfying, 
$$c(x)= a(x)+b(x)>c$$
for a constant $c>0$. Then $(P+\lambda^2+\lambda c(\cdot))^{-1}$ exists for pure imaginary $\lambda=i\tau$, $\tau\in \R\backslash\{0\}$ and we have
$$(P+\lambda^2+\lambda a(\cdot))(P+\lambda^2+\lambda c(\cdot))^{-1} =Id-  \lambda b(\cdot)(P+\lambda^2+\lambda c(\cdot))^{-1} $$
Since $(P+\lambda^2+\lambda c(\cdot))^{-1} $ takes $L^2$ to $H^1$ and $b(x)$ is a compact support function, by  Rellich's compact embedding theorem, $b(\cdot)(P+\lambda^2+\lambda c(\cdot))^{-1} $ is a compact operator. Thus by the Fredholm theory, the resolvent exists if and only if the following equation have a solution
$$ \lambda b(\cdot)(P+\lambda^2+\lambda c(\cdot))^{-1}u=u.$$
Taking $u=(P+\lambda^2+\lambda c(\cdot))v$, the equation can be written as follows
$$ (P+\lambda^2+\lambda a(x))v=0.$$
For pure imaginary $\lambda=i\tau$, $\tau\in \R\backslash\{0\}$, multiplying $\bar v$ to the equation, integrating on $\R^d$ and taking imaginary part, we have
$$\int_{\R^d} a(x)|v(x)|^2 dx=0.$$
So $v=0$ on the support of $a$. Then by the unique continuation of second order elliptic operators, we conclude $v=0$. Thus for pure imaginary $\lambda\not=0$, we have $u=0$. So the resolvent exist and by continuity, for middle frequency, we have the existence of the resolvent across the imaginary axis. Combining the existence of the resolvent for high and middle frequency, we have the following lemma. 
\begin{lemma}
For any $\varepsilon>0$ there exists $\delta>0$ such that $R(\lambda)=(P+\lambda^2+\lambda a(\cdot))^{-1}:L^2 \to L^2$ exists for ${\rm Re}\lambda>-\delta$ if $|{\rm Im}\lambda|>\varepsilon$. 
\end{lemma}
From these lemmas, we can conclude the energy decay by applying similar argument as in section $4$ c.f. Remark $\ref{ene rem}$.
\begin{remark}
We can also prove the energy decay only assuming  the geometric control condition, by using the argument in \cite{Jo}, see Appendix. 
\end{remark}

\begin{remark}
We can also treat a small perturbation of the above case and we can get the energy decay for the case that $a(x)\geq 0$ does not hold. In the above setting, we take a smooth function $b(x)>0$ with bounded derivatives and we consider the following damped wave equation.
$$(\pl_{t}^2+P+a(x)\pl_t-\varepsilon b(x)\pl_t  )u(t,x)=0.$$
From the above argument, we get the following bound for $\lambda=i\tau$, $\tau\in\R\backslash\{0\}$
$$\| (P+\lambda^2+\lambda a(\cdot))^{-1}u\|_{L^2}\leq C|\lambda|^{-1}\|u\|_{L^2}.$$
Taking $u=(P+\lambda^2+\lambda a(\cdot))v$, we get
$$|\lambda|\| v\|_{L^2}\leq C \|(P+\lambda^2+\lambda a(\cdot))v\|_{L^2}.$$
So we have
\begin{align*}
\|(P+\lambda^2+\lambda a(\cdot)-\varepsilon\lambda b(\cdot))v\|_{L^2}&\geq \|(P+\lambda^2+\lambda a(\cdot))v\|_H-C\varepsilon|\lambda| \|v\|_{L^2}\geq C |\lambda| \|v\|_{L^2}
\end{align*}
if $\varepsilon$ is sufficiently small. This inequality implies the existence of the resolvent and we can apply previous argument to get the energy decay though for applying the previous argument, we need a small modification since at first we only know  $U(t)$ may be growth like $e^{C\varepsilon t}$ order.
\end{remark}

We also give the asymptotic profile of the solution. We take a smooth compact support function $b(x)$ satisfying
$$c(x)= a(x)+b(x)>c$$
for a constant $c>0$. We take 
$$\tilde P:=c(\cdot)^{-1/2} P c(\cdot)^{-1/2}.$$
The following theorem is a generalization of Theorem ${\ref{main}}$. 
\begin{theorem}\label{geom asym}
Let $u$, $v$ be the solution  of the following equations
\begin{align}\label{damp cauchy}
&\begin{cases}
&(\pl_t^2+P+a(\cdot)\pl_t)u=0,\\
& u|_{t=0}=u_0, \pl_t u|_{t=0}=u_1,
\end{cases}\\
&\begin{cases}
&(\pl_t+ \tilde P )v=0,\\
& v|_{t=0}=c(\cdot)^{1/2}u_0+c(\cdot)^{-1/2}u_1
\end{cases}\label{heat cauchy}
\end{align}
on $\R \times \R^d$ for $d>2$, respectively.
Under the assumption of Theorem $\ref{geo ene decay}$, there exists $C>0$ such that for any $u_0\in H^1$, $u_1\in L^2$ and $t>1$, the following asymptotic profile holds
\begin{align*}
\|u(t)-c(\cdot)^{1/2}v(t)\|_{L^2}\leq C{t^{-1}}(\|u_0\|_{H^1}+\|u_1\|_{L^2} ).
\end{align*}
\end{theorem}
To prove the theorem, we need the asymptotic expansion of the resolvent near the origin. Since $c(x)$ is uniformly positive, the resolvent $R(\lambda)=(P+\lambda^2+\lambda c(\cdot) )^{-1}$ exists near the imaginary axis. We have the following identity
\begin{equation}\label{res-com}
(P+\lambda^2+\lambda c(\cdot) )^{-1}(P+\lambda^2+\lambda a(\cdot))=Id+ \lambda(P+\lambda^2+\lambda c(\cdot))^{-1}b(\cdot). 
\end{equation}
We use the following estimate for the cut-off resolvent. 
\begin{lemma}
Let $d>2$ and $\chi$ is a continuous compact support function then there exists a constant $\delta>0$ such taht the following estimate holds for $|\lambda|$ with ${\rm Re} \lambda>0$ or $|{\rm Im}\lambda|>\delta|{\rm Re}\lambda|$,  
$$\|\chi R(\lambda)\chi\|_{L^2\to L^2}={\cal O}(1).$$
\end{lemma}
\begin{proof}
By uniform ellipticity, we have
$$|((P+\lambda^2+\lambda c(\cdot) )u,u)_{L^2}|\geq C\|\nabla u\|^2_{L^2}+C|\lambda| \int_{\R^d} c(x)|u(x)|^2 dx-C|\lambda|^2\|u\|_{L^2}^2$$
for ${\rm Re}\lambda>0$ or $|{\rm Im}\lambda|>\delta |{\rm Re}\lambda|$ if $\delta$ is sufficient small. 
Since $c(x)>c$, for sufficient small $\lambda$, we have
$$|\lambda| \int_{\R^d} c(x)|u(x)|^2 dx-|\lambda|^2\|u\|_{L^2}^2\geq 0.$$
If $d>2$, by the Hardy's inequality, we obatain
$$\|\nabla u\|^2_{L^2}\geq C\||x|^{-1}u\|_{L^2}^2\geq C\|\langle x\rangle^{-1}u\|_{L^2}^2.$$
Here $\langle x\rangle =\sqrt{1+x^2}$. So we have
\begin{align*}
\|(\langle x\rangle (P+\lambda^2+\lambda c(x) )u\|_{L^2}\| \langle x\rangle^{-1} u\|_{L^2}&\geq |(\langle x\rangle (P+\lambda^2+\lambda c(x) )u, \langle x\rangle^{-1} u)_{L^2}|\\
&=|(P+\lambda^2+\lambda c(x))u,u)_{L^2}|\geq C\|\langle x\rangle^{-1}u\|_{L^2}^2.
\end{align*}
Taking $u=R(\lambda)\langle x\rangle^{-1}v$, we obtain
\begin{equation}\label{cut-off resolvent}
\|\langle x\rangle^{-1} R(\lambda)\langle \cdot \rangle^{-1}v\|_{L^2}\leq \frac{1}{C}\|v\|_{L^2}.
\end{equation}
From this inequality, we have the estimate
\end{proof}
For $d=1,2$, we can give the following results.
\begin{lemma}
If $\chi$ is a real-valued continuous compact support function, then there exists a constant $\delta>0$ such that the following estimates hold for sufficient small $|\lambda|$ with ${\rm Re} \lambda>0$ or $|{\rm Im}\lambda|>\delta|{\rm Re}\lambda|$,  
\begin{align*}
&\|\chi R(\lambda)\chi\|_{L^2\to L^2}={\cal O}(\log|\lambda|) &(d=2),\\
&\|\chi R(\lambda)\chi\|_{L^2\to L^2}={\cal O}(|\lambda|^{-1/2}) &(d=1).
\end{align*}
\end{lemma}
\begin{proof}
By uniform ellipticity, we have
\begin{equation}\label{est1}
|((P+\lambda^2+\lambda c(\cdot) )u,u)_{L^2}|\geq C(\|\nabla u\|^2_{L^2}+|\lambda|\|u\|_{L^2}^2)\geq C \|(|\nabla|+|\lambda|^{1/2})u\|^2_{L^2}.
\end{equation}
We prove the following estimate.
\begin{equation}\label{est2}
\|(|\nabla|+|\lambda|^{1/2})u\|^2_{L^2}\geq 
\begin{cases}
&C\frac{1}{\log|\lambda|}\|\chi u\|_{L^2}^2, \quad(d=2),\\
&C|\lambda|^{1/2}\|\chi u\|_{L^2}^2, \quad(d=1).
\end{cases}
\end{equation}
We take a compact support cut off function $\tilde \chi(\xi)$ which satisfies $\tilde \chi(\xi)=1$ if $|\xi|$ is sufficient small. Then by the Plancherel theorem and $(|\xi|+|\lambda|^{1/2})^{-1}(1-\tilde\chi(\xi))$ is bounded for sufficient small $\lambda$, we have 
\begin{align*}
\|\chi (|\nabla|+|\lambda|^{1/2})^{-1}u\|^2_{L^2}&\leq\|\chi (|\nabla|+|\lambda|^{1/2})^{-1}{\tilde \chi}(D)u\|^2_{L^2}+ \|(|\nabla|+|\lambda|^{1/2})(1-\tilde\chi(D))u\|^2_{L^2}\\
&\leq \|(|\nabla|+|\lambda|^{1/2})^{-1}{\tilde \chi}(D)u\|^2_{L^\8}+ \|(|\xi|+|\lambda|^{1/2})^{-1}(1-\tilde\chi(\xi))\widehat u\|^2_{L^2}\\
&\leq \|(|\xi|+|\lambda|^{1/2})^{-1}{\tilde \chi}(\xi)u\|^2_{L^1}+C \|u\|^2_{L^2}\\
&\leq \|(|\xi|+|\lambda|^{1/2})^{-1}{\tilde \chi}(\xi)\|_{L^2}^2 \|u\|^2_{L^2}+C \|u\|^2_{L^2}.
\end{align*}
By easy computations, we obtain
$$\|(|\xi|+|\lambda|^{1/2})^{-1}{\tilde \chi}(\xi)\|_{L^2}^2\leq 
\begin{cases}
&C\log|\lambda|, \quad(d=2),\\
&C|\lambda|^{-1/2}, \quad(d=1).
\end{cases}
$$
Taking $u=(|\nabla|+|\lambda|^{1/2})v$, we have $(\ref{est2})$. From $(\ref{est1})$ and $(\ref{est2})$, we get
\begin{equation*}
\|\chi u\|_{L^2}^2\leq 
\begin{cases}
&C\log|\lambda||((P+\lambda^2+\lambda c(x) )u,u)_{L^2}|, \quad(d=2),\\
&C|\lambda|^{-1/2}|((P+\lambda^2+\lambda c(x) )u,u)_{L^2}|, \quad(d=1).
\end{cases}
\end{equation*}
Now we take $u=(P+\lambda^2+\lambda c(x) )^{-1} \chi v$ then we have 
\begin{equation*}
\|\chi u\|_{L^2}\|\chi (P+\lambda^2+\lambda c(x) )^{-1} \chi v\|_{L^2}\leq 
\begin{cases}
&C\log|\lambda|{|((v,\chi u)_{L^2}\leq C  \log|\lambda|}\|v\|_{L^2}\|\chi u\|_{L^2}, \quad(d=2),\\
&C|\lambda|^{-1/2}|((v,\chi u)_{L^2}\leq C|\lambda|^{-1/2}\|v\|_{L^2}\|\chi u\|_{L^2} , \quad(d=1).
\end{cases}
\end{equation*}
Dividing by $\|\chi u\|_{L^2}$, we have proved.
\end{proof}
We take the inverse of the right hand side of the identity $(\ref{res-com})$. 
We define
$$\tilde R(\lambda)=\sqrt{b(\cdot)}\lambda(P+\lambda^2+\lambda c(\cdot))^{-1}\sqrt{b(\cdot)}.$$
Since $b(x)$ is compact support function, by the above lemmas, we can assume $\|\tilde R(\lambda)\|_{L^2 \to L^2}<1$ for sufficient small $\lambda$. For such $\lambda$, we can construct the inverse by Neumann series.
\begin{align*}
&(Id- \lambda(P+\lambda^2+\lambda c(\cdot))^{-1}b(\cdot))^{-1}\\
&=Id+\sum_{n=0}^\8 (-1)^n\lambda(P+\lambda^2+\lambda c(\cdot) )^{-1}\sqrt{b(\cdot)}\{\lambda\tilde R(\lambda)\}^{n}\sqrt{b(\cdot)}.
\end{align*}
Thus we have the following identity for the resolvent $R(\lambda)=(P+\lambda^2+\lambda a(\cdot))^{-1}$ and $R_c(\lambda)=(P+\lambda^2+\lambda c(\cdot))^{-1}$
\begin{equation}\label{res-exp}
\begin{split}
R(\lambda)&=(P+\lambda^2+\lambda c(\cdot))^{-1}\\
&\ \ +\sum_{n=0}^\8(-1)^n \lambda(P+\lambda^2+\lambda c(\cdot))^{-1}\sqrt{b(\cdot)}\{\lambda\tilde R(\lambda)\}^{n}\sqrt{b(\cdot)}(P+\lambda^2+\lambda c(\cdot))^{-1}\\
&=R_c(\lambda)+\lambda R_c(\lambda){b(\cdot)}R_c(\lambda)+R_1(\lambda).
\end{split}
\end{equation}
Here $R_1(\lambda)$ is the reminder term. So we have the asymptotic expansion of the resolvent. We shall prove Theorem $\ref{damp cauchy}$. If $d>2$ then $\tilde R(\lambda)$ is bounded so we have $\|\lambda \tilde R(\lambda)\|_{L^2\to L^2}={\cal O}(|\lambda|)$ and $\|R_1(\lambda)\|_{l^2\to L^2}={\cal O}(1)$ near the origin. To the first term $R_c(\lambda)$, we can apply similar argument in Theorem ${\ref{main}}$ and get the heat type asymptotic. $R_1(\lambda)$ term becomes lower order term since $\|R_1(\lambda)\|={\cal O}(1)$ is lower than the main order $|\lambda|^{-1}$ and we can get $t^{-1}$ order decay.
The main difference is the second term and we study the term. We have $\|\phi_{c}(P)\lambda R_c(\lambda){b(x)}R_c(\lambda)\|={\cal O}(1)$ if $\phi_c(P)$ cut low frequency so, corresponding to Proposition \ref{high}, we have $t^{-1}$ order decay to the high frequency part. For low frequency part, we follow the proof of Theorem {\ref{main}}. To obtain $(\ref{asy})$, we do not use the identity $(\ref{asym resolvent})$ and we only use these estimate of the resolvent which is similar in this situation. Thus the main term is $\tilde J_3$ and we have
\begin{equation}\label{asy2}
\begin{split}
u(t)=\frac{1}{2\pi}\int_{{\cal C}}\tilde \phi_{c}(P)e^{i\lambda t}\left(R(i\lambda)(i\lambda+c(\cdot))u_0+R(i\lambda)u_1\right)d\lambda+{\cal O} (t^{-1}).
\end{split}
\end{equation}
Inserting $(\ref{res-exp})$, since $R_1(\lambda)$ and the first term are a lower order, we get 
\begin{equation}
\begin{split}
u(t)&=\frac{1}{2\pi}\int_{{\cal C}}\tilde \phi_{c}(P)e^{i\lambda t}\left(R_c(i \lambda)+i \lambda R_c(i \lambda){b(\cdot)}R_c(i \lambda))(c(\cdot)u_0+u_1\right)d\lambda+{\cal O} (t^{-1})\\
&=L_1+\frac{1}{2\pi}L_2+{\cal O} (t^{-1}).
\end{split}
\end{equation}
In the above integral the first term is similar to the strictly positive case, so we can get the following asymptotic profile by applying similar argument to prove Theorem {\ref{main}}. 
$$L_1=c(\cdot)^{-1/2}e^{-t \tilde P} c(\cdot)^{-1/2}(c(\cdot)u_0+u_1)+{\cal O} (t^{-1}).$$
To estimate $L_2$ part, we recall the following properties of our heat propagator.
\begin{lemma}\label{heat est}
There exists a constant $C>0$ such that the following bounds hold for $u\in L^1\cap L^2$ and $t>1$
\begin{align*}
&\|e^{-t \tilde P}u\|_{L^1}\leq C \|u\|_{L^1},\\
&\|e^{-t \tilde P}u\|_{L^2}\leq C t^{-d/4}\|u\|_{L^1},\\
&\|e^{-t \tilde P}u\|_{L^\8}\leq C t^{-d/4}\|u\|_{L^2}.
\end{align*}
\end{lemma}
\begin{proof}
Probably this result is well know c.f. \cite{Da} and we give an outline of the proof. We use the following Nash's inequality 
$$\|u\|_{L^2}^{2+4/d}\leq C(\tilde Pu, u)_{L^2}\|c(\cdot)^{1/2}u\|_{L^1}^{2/d}.$$
Since $P$ is uniformly elliptic and $c(\cdot)$ is strictly positive  bounded function, this inequality is easily obtained from the following usual Nash's inequality
$$\|u\|_{L^2}^{2+4/d}\leq C\|\nabla u\|^2_{L^2}\|u\|_{L^1}^{4/d}.$$
Let $u(t)=u(t,\cdot)=e^{-t\tilde P}u_0$, then by the integration by parts, we have
$$\frac{d}{dt}\int_{\R^d} c(x)^{1/2}u(t,x)dx=-\int_{\R^d}\pl_i\left(g_{ij}(x)\pl_j c(x)^{-1/2}u(t,x)\right)dx=0.$$
Thus the propagator $e^{t \tilde P}$ preserves this modified total heat. Under our assumption, $e^{t \tilde P}$ also preserve positivity see \cite{Fr}. So we can only consider the positive solution and we assume $u(t)$ is positive. From the conservation of the modified total heat, we have the $L^1L^1$ bound. Next we define
$$H(t):=\frac{\|u(t)\|_{L^2}^2}{\|c(\cdot)^{1/2}u(t)\|_{L^1}^2}.$$
Since the modified total heat is preserved, by the Nash's inequality, we have
$$\frac{d}{dt}H(t)=-\frac{2(\tilde Pu,u)}{\|c(\cdot)^{1/2}u(t)\|_{L^1}^2}\leq- C H(t)^{1+2/d}.$$
This inequality implies
$$-H(t)^{-2/d} \leq H(0)^{-2/d}-H(t)^{-2/d}\leq -Ct.$$
So we obtain
$$\|u(t)\|_{L^2}\leq Ct^{-d/4} \|u_0\|_{L^1}.$$
From this estimate, we can prove the $L^\8$ estimates by using duality argument. 
For positive $v$ and $s\leq t$, we take
$$v(s,x)=v(s):=e^{(s-t )\tilde P}v.$$
Then we have
\begin{align*}
\frac{d}{ds}\int_{\R^d} u(s,x)v(s,x)dx=\int_{\R^d}\left (-\tilde P u(s,x)\right)v(s,x)dx+\int_{\R^d} u(s,x)\left(\tilde P v(s,x)\right)dx=0.
\end{align*}
Thus the above integral does not depend $s$ so taking $s=t$ and $s=t/2$, we have
$$\int_{\R^d} u(t,x)vdx=\int_{\R^d} u(t/2,x)v(t/2,x)dx\leq \|u(t/2)\|_{L^2}\|v(t/2)\|_{L^2}\leq C 
 t^{-d/4}\|u\|_{L^2}\|v\|_{L^1}.
 $$ 
By this inequality and duality, we get the $L^\8L^2$ estimate. 
\end{proof}
Using this bound, we estimate $L_2$ in the decomposition $(\ref{asy2})$. From the identity $(\ref{asym resolvent})$, we have 
\begin{equation}\label{asym resolvent2}
R_c(\lambda)=c(\cdot)^{-1/2}(\lambda+\tilde P)^{-1}c(\cdot)^{-1/2}-\lambda^2c(\cdot)^{-1/2}(\lambda+\tilde P)^{-1}c(\cdot)^{-1/2}R_c(\lambda).
\end{equation}
Inserting this identity to $L_2$, since $\lambda^2c(\cdot)^{-1/2}(\lambda+\tilde P)^{-1}c(\cdot)^{-1/2}R_c(\lambda)={\cal O}(|\lambda|)$, we obtain
\begin{align*}
L_2&= \int_{{\cal C}}\tilde \phi_{c}(P)e^{i\lambda t} i \lambda c(\cdot)^{-1/2}(\lambda+\tilde P)^{-1}c(\cdot)^{-1/2}{b(x)}c(\cdot)^{-1/2}(\lambda+\tilde P)^{-1}c^{-1/2}(c(\cdot)u_0+u_1)d\lambda+{\cal O} (t^{-1})\\
&=L_3+{\cal O} (t^{-1}).
\end{align*}
So we estimate $L_3$ part. We use identities $(\ref{re} )$ and $(\ref{re2} )$. By these identities, we have
\begin{align*}
L_3&= \int_{{\cal C}}\tilde \phi_{c}(P) i \lambda c(\cdot)^{-1/2}(\lambda+\tilde P)^{-1}e^{-t\tilde P/2}c(\cdot)^{-1/2}{b(\cdot)}c(\cdot)^{-1/2}e^{-t\tilde P/4}(i\lambda+\tilde P)^{-1}c(\cdot)^{-1/2}(c(\cdot)u_0+u_1)d\lambda\\
&+\int_{{\cal C}}\tilde \phi_{c}(P)e^{i\lambda t} i \lambda c^{-1/2}\int^{t/2}_0 e^{-t_1(i\lambda+\tilde P)}dt_1c^{-1/2}{b(\cdot)}c(\cdot)^{-1/2}(i\lambda+\tilde P)^{-1}e^{-t(i\lambda+\tilde P)/4}c^{-1/2}(c(\cdot)u_0+u_1)d\lambda\\ 
&+\int_{{\cal C}}\tilde \phi_{c}(P)e^{i\lambda t} i \lambda c(\cdot)^{-1/2}(\lambda+\tilde P)^{-1}e^{-t\tilde P/2}c(\cdot)^{-1/2}{b(\cdot)}c(\cdot)^{-1/2}\int^{t/4}_0 e^{-t_2(i\lambda+\tilde P)}dt_2c(\cdot)^{-1/2}(c(\cdot)u_0+u_1)d\lambda\\
&+\int_{{\cal C}}\tilde \phi_{c}(P)e^{i\lambda t} i \lambda c(\cdot)^{-1/2}\int^{t/2}_0 e^{-t_1(i\lambda+\tilde P)}dt_1c(\cdot)^{-1/2}{b(\cdot)}c(\cdot)^{-1/2}\int^{t/4}_0 e^{-t_2(i\lambda+\tilde P)}dt_2c(\cdot)^{-1/2}(c(\cdot)u_0+u_1)d\lambda\\
&=L_4+L_5+L_6+L_7
\end{align*}
Since $b(x)$ is compact support function and $d>2$, from Lemma $\ref{heat est}$,  we have
\begin{align*}
\|e^{-t\tilde P/2}c(x)^{-1/2}{b(x)}c(x)^{-1/2}e^{-t\tilde P/4}\|_{L^2\to L^2}&\leq \|e^{-t\tilde P/2}\|_{L^1\to L^2}\|c(x)^{-1/2} {b(x)}c(x)^{-1/2}e^{-t\tilde P/4}\|_{L^2\to L^1}\\
&\leq C t^{-d/4} \|c(x)^{-1/2} {b(x)}c(x)^{-1/2}\|_{L^1}\|e^{-t\tilde P/4}\|_{L^2\to L^\8}\\
&\leq C t^{-d/2}\leq C t^{-1}.
\end{align*}
From this estimate, we know $L_4$ is a lower order term. For $L_5$, since 
$$\left\|\int^{t/2}_0 e^{-t_1\tilde P}dt_1\right\|_{L^1\to L^2}=\max\{1, t^{-d/4+1+\varepsilon}\}$$
for any $\varepsilon>0$, we have
\begin{align*}
&\left\|\int^{t/2}_0e^{i\lambda(3t-4t_1)/4}( i \lambda c^{-1/2} e^{-t_1\tilde P}c^{-1/2}{b(\cdot)}c(\cdot)^{-1/2}(i\lambda+\tilde P)^{-1}e^{-t\tilde P/4}c^{-1/2}dt_1\right\|_{L^2\to L^2}\\
&\leq \int^{t/2}_0\left\|e^{i\lambda(3t-4t_1)/4}  c^{-1/2} e^{-t_1\tilde P}\right\|_{L^1\to L^2}\left\|c^{-1/2} b(\cdot) c(\cdot)^{-1/2}e^{-t\tilde P/4}i \lambda(i\lambda+\tilde P)^{-1}c^{-1/2}\right\|_{L^2\to L^1}dt_1\\
&\leq \int^{t/2}_0\left\|e^{i\lambda(3t-4t_1)/4}c^{-1/2} e^{-t_1\tilde P}\right\|_{L^1\to L^2}dt_1\left\|c^{-1/2} b(\cdot) c(\cdot)^{-1/2}\|_{L^1}\|e^{-t\tilde P/4}\|_{L^2\to L^{\8}} \|i \lambda(i\lambda+\tilde P)^{-1}c^{-1/2}\right\|_{L^2\to L^2}\\
&\leq C \max\{ e^{-3t{\rm Im} \lambda/4},  e^{-t{\rm Im} \lambda/4}\}\max\{1, t^{-d/4+1+\varepsilon}\} t^{-d/4}.
\end{align*}
From this estimate the integrand is ${\cal O}(1)$ as $\lambda$ variable and by contour deformation as in the proof of Proposition $\ref{low decay}$, we get further $t^{-1}$ decay. Thus we have
$$\|L_5\|_{L^2}\leq C \max\{1, t^{-d/4+1+\varepsilon}\} t^{-d/4-1}\leq Ct^{-1}.$$
We can apply similar argumetn to estimate $L_6$ term and we get $t^{-1}$ order decay. For $L_7$ term, the integrand is holomorphic so thanks to the $e^{-i\lambda (t-t_1-t_2)}$ factor, we can get the exponential decay estimate, by contour deformation across the origin. 
Thus we have proved Theorem $\ref{geom asym}$. For $d=1,2$, in the same way, we can also prove the following result.
\begin{theorem}
Under the assumption of Theorem $\ref{geo ene decay}$, for any $\varepsilon>0$ and $d=1,2$, there exists $C>0$ such that for $t>2$ the following estimates between the solution $u$  of the equation $(\ref{damp cauchy})$ and the solution $v$ of the equation $(\ref{heat cauchy})$ hold
\begin{align*}
&\|u(t)-c(\cdot)^{1/2}v(t)\|_{L^2}\leq C{t^{-1}}\log t (\|u_0\|_{H^1}+\|u_1\|_{L^2} ), &(d=2),\\
&\|u(t)-c(\cdot)^{1/2}v(t)\|_{L^2}\leq C{t^{-1/2}}(\|u_0\|_{H^1}+\|u_1\|_{L^2} ), &(d=1). 
\end{align*}
\end{theorem}
\begin{remark}
In the above theorems we have the arbitrariness to the choice of $c(\cdot)$ term. Probably our estimates come from the arbitrariness of $c(\cdot)$. 
For simplicity, we consider the following heat equation
$$((1+a(\cdot))\pl_t-\Delta)u=0,\quad u|_{t=0}=(1+a(\cdot))^{-1}u_0$$
Here $a(\cdot)$ is a non-negative compact support function. Then by Duhamel's principle we have
$$u(t)=e^{t\Delta}u(0)-\int^t_0e^{(t-s)\Delta }a(\cdot)\pl_s u(s)ds.$$
By integration by parts, we obtain
$$u(t)=e^{t\Delta}u_0-\int^t_{t/2}e^{(t-s)\Delta }a(\cdot)\pl_s u(s)ds- \int^{t/2}_{0}\left(\pl_te^{(t-s)\Delta }\right)a(\cdot)u(s)ds-e^{t/2\Delta}a(\cdot) u(t/2).$$
In the right hand side, the first term is the top term. Since $a(x)$ is a compact support function, we have
\begin{align*}
\|e^{t\Delta/2}a(\cdot) u(t/2)\|_{L^2}&\leq \|e^{t/2\Delta}\|_{L^1\to L^2}\|a(\cdot)u(t/2)\|_{L^1}\\
&\leq \|e^{t/2\Delta}\|_{L^1\to L^2}\|a(\cdot)\|_{L^1} \|u(t/2)\|_{L^2\to L^\8}\leq Ct^{-d/2}\|u(0)\|_{L^2}.
\end{align*}
From well-known estimates of heat propagators, we get the following estimate 
$$\left\|\int^t_{t/2}e^{(t-s)\Delta }a(\cdot)\pl_s u(s)ds\right\|_{L^2}\leq C\int^t_{t/2} (t-s)^{-d/4}s^{-d/4-1}\|a(\cdot)\|_{L^1}\|u(0)\|_{L^2}  ds \leq C 
\begin{cases}
t^{-1}\|u_0\|_{L^2} &(d=2),\\
t^{-1/2} \|u_0\|_{L^2} &(d=1).
\end{cases} $$
We can treat the another integral by same manner and get similar estimate. So the above estimate seems almost optimal.
\end{remark}
In the same manner, we can generalize Theorem ${\ref{main 2}}$.

\begin{theorem}\label{geom asym2}
Let $u$, $v$ are the solution of the equations $(\ref{damp cauchy})$, $(\ref{heat cauchy})$ respectively. Then 
under the assumption of Theorem 7.1, for any $\varepsilon>0$, there exists $C>0$ such that for any $u_0\in H^1\cap L^1$ and $\langle x\rangle u_0\in L^2$, $u_1\in L^2\cap L^1$  and $t>2$, the following asymptotic profile hold

$$\|u(t)-c(\cdot)^{1/2}v(t)\|_{L^2}\leq C{t^{-1-d/4}}(\|u_0\|_{H^1\cap L^1 }+\|u_1\|_{L^2\cap L^1} )$$
\end{theorem}
\begin{proof}
The proof is essentially same and we give an outline of the proof. We use the expansion $(\ref{res-exp})$. $R_c$ term can be treated by similar manner in the proof of Theorem $\ref{main 2}$ and we get the heat type asymptotic. 
\end{proof}
\end{section}
\begin{section}{Appendix}
In this appendix, we prove the energy decay estimate under the geometric control condition. We estimate the resolvent dividing three frequency parts in the same way as Theorem \ref{geo ene decay}. We can use Lemma $\ref{highreso}$ to estimate the high frequency part. So we need to estimate middle and low frequency parts. We take
$$L=P+\lambda^2+\lambda a(\cdot)$$
and estimate $\|Lu\|_{L^2}$ from below. The following argument is essentially similar to \cite{Jo}. By the geometric control condition, we can take a sequence $(x_n)\subset\R^d$ such that $a(x_n)\geq \alpha/T$ and any point of $\R^d$ is at bounded distance from the set $\cup_n\{x_n\}$. By the uniformly continuity, we can find $r,\beta>0$ such that $a(x)\geq \beta>0$ on $\omega=\cup_n B(x_n,r)$. 

We estimate the low frequency part. First we assume $|{\rm Re}\lambda|< |{\rm Im}\lambda|$. By taking imaginary part of $(Lu,u)_{L^2}$, we easily have
\begin{equation}\label{first}
\|Lu\|^2_{L^2}+\varepsilon^2 |\lambda|^2\|u\|^2_{L^2}+C\varepsilon |\lambda|^3\|u\|^2_{L^2} \geq \frac{\varepsilon |\lambda|^2}{C}\|\sqrt a u\|^2_{L^2}.
\end{equation}
We can take $\varepsilon>0$ arbitrary small and $C>0$ is some constant. As in \cite{Jo}, we set
$$Q_h=h^2e^{\varphi/h}(P+\lambda^2)e^{-\varphi/h}$$
here $h>0$ is sufficiently small.
Then we can take a uniformly bounded weight $\varphi$ such that the following sub-elliptic estimate holds
\begin{equation}\label{sub-elliptic}
\|Q_hu\|^2_{L^2}\geq Ch\|u\|^2_{L^2}
\end{equation}
here $C>0$ and $u$ satisfies $u|_\omega=0$, see \cite{Jo} for the proof. We take a real-valued cut-off function $\chi\in C^\8_0$ such that $\chi(s)=0$ for $s\geq\beta$ and $\chi\equiv 1$ in a neighborhood of $0$. Then $a^c =\chi\circ a$ vanishes on $\omega$ and any derivatives of  $a^c$ is controlled by $a$. We set $v=e^{\phi/h} a^cu$ and compute
$$Q_hv=h^2e^{\phi/h}a^c Lu-h^2e^{\phi/h}\lambda a^c au+ h^2e^{\phi/h}\left[P, a^c\right] u.$$
Since $v|_\omega=0$, from this identity and $(\ref{sub-elliptic})$,  we obtain
$$\|h^2e^{\phi/h}a^c Lu\|^2_{L^2}+\|h^2e^{\phi/h}\lambda a^c  au\|^2_{L^2}+ \|h^2e^{\phi/h}\left[P, a^c \right] u\|^2_{L^2}\geq \|Q_hv\|^2_{L^2}\geq Ch\|v\|_{L^2}^2.$$
Since $\varphi$ is uniformly bounded, for fixed $h$, $e^{\phi/h}$ is uniformly positive and bounded . Thus we have
$$\|Lu\|^2_{L^2}+C|\lambda|^2\|au\|^2_{L^2}+ C\|\left[P, a^c\right] u\|^2_{L^2}\geq \frac{1}{C}\|a^cu\|_{L^2}^2.$$
Since derivatives of $a^c$ are controlled by $a$, for sufficient small $|\lambda|$, we have
\begin{equation}\label{second}
\varepsilon^{3/2}|\lambda|^2 \|Lu\|^2_{L^2}+C\varepsilon^{3/2}|\lambda|^2\|{\sqrt a}u\|^2_{L^2}+ C\varepsilon^{3/2}|\lambda|^2 \sum_{i,j=1}^d\|(\pl_j  a^c )\pl_iu\|^2_{L^2}\geq \frac{\varepsilon^{3/2 }|\lambda|^2}{C}\|a^cu\|_{L^2}^2.
\end{equation}
For $k=1,2,\cdots,d$, using integration by parts, we have
$$(Pu, (\pl_k a^c)^2u)_{L^2}=\sum_{i,j=1}^d(g_{ij}(\pl_k a^c)\pl_iu, (\pl_k a^c)\pl_j u)_{L^2}+2\sum_{i,j=1}^d((\pl_k a^c)\pl_iu, g_{ij}(\pl_j\pl_k a^c)u)_{L^2}.$$
Since derivatives of $a^c$ are controlled by $a$, from this identity, we have 
$$\|Lu\|_{L^2}^2+\|\sqrt au\|^2_{L^2}\geq C \sum_{k,i=1}^d\|(\pl_k  a^c )\pl_iu\|^2_{L^2}$$
for sufficiently small $|\lambda|$. So from $(\ref{second})$, we get
\begin{equation}\label{third}
\varepsilon^{3/2}|\lambda|^2 \|Lu\|^2_{L^2}+C\varepsilon^{3/2}|\lambda|^2\|{\sqrt a}u\|^2_{L^2}\geq \frac{\varepsilon^{3/2 }|\lambda|^2}{C}\|a^cu\|_{L^2}^2.
\end{equation}
Thus from $(\ref{first})$ and $(\ref{third})$, we obtain
$$\|Lu\|^2_{L^2}+C\varepsilon^2 |\lambda|^2\|u\|^2_{L^2}+C\varepsilon^{3/2}|\lambda|^2\|{\sqrt a}u\|^2_{L^2}+C\varepsilon |\lambda|^3\|u\|^2_{L^2} \geq \frac{\varepsilon |\lambda|^2}{C}\|\sqrt a u\|^2_{L^2}+\frac{\varepsilon^{3/2 }|\lambda|^2}{C}\|a^cu\|_{L^2}^2.$$
Now taking $\varepsilon$ sufficiently small and $|\lambda|<\varepsilon$, since $a^c(x)=1$ if $a(x)=0$, we conclude
\begin{equation}\label{forth}
\|Lu\|^2_{L^2}\geq C|\lambda|^2\|u\|_{L^2}^2.
\end{equation}
From this estimate, we get $R(\lambda)={\cal O}(|\lambda|^{-1})$ for sufficiently small $|\lambda|$ around the imaginary axis. If ${\rm Re}\lambda \geq|{\rm Im}\lambda|$, by taking real part of $(Lu,u)_{L^2}$, we have
$$\|Lu\|^2_{L^2}+\varepsilon^2 |\lambda|^2\|u\|^2_{L^2}\geq C\varepsilon |\lambda|^2\|\sqrt a u\|^2_{L^2}. $$
By similar argument, we can also prove $(\ref{third})$ in this case. So we have $(\ref{forth})$ and get $R(\lambda)={\cal O}(|\lambda|^{-1})$ for ${\rm Re} \lambda>0$. Thus we have estimated the low frequency part. The estimate of the middle frequency part is essentially same and more easy (we only need the reversibility of $P-\lambda^2+i\lambda a$ if $\lambda$ is real and in a compact subset of $(0,\8)$) so we omit the proof. From these estimate, We conclude the following energy decay estimate. 
\begin{theorem}\label{app ene decay}
We assume the geometric control condition. Let $u$ be a solution of $(\ref{euc damp})$ with the initial data $u|_{t=0}=u_0\in H^1$ and $\pl_tu|_{t=0}=u_1\in L^2$. 
Then there exists a constant $C$ such that for $t>1$ the following bound holds
$$\|\nabla u\|^2_{L^2}+\|\pl_t u\|^2_{L^2}\leq C(\sum_{i,j=1}^d(g_{ij}\pl_iu,\pl_ju)_{L^2}+\|\pl_t u\|^2_{L^2}) \leq C t^{-1}(\|u_0\|^2_{H^1}+\|u_1)\|_{L^2}).$$
\end{theorem}
\end{section}
\begin{small}

\begin{center}

\end{center}

Department of Education, Wakayama University,

930, Sakaedani, Wakayama city , Wakayama 640-8510, Japan 

E-mail:h2480@center.wakayama-u.ac.jp
\end{small}
 \end{document}